\documentclass[11pt,a4paper]{article}
\usepackage[margin=25mm]{geometry}
\usepackage[T1]{fontenc}
\usepackage{lmodern}
\usepackage{amsthm}
\usepackage{authblk}
\usepackage[labelfont=bf]{caption}
\usepackage{subcaption}
\usepackage{bigints}
\usepackage{abraces}
\usepackage{bbold}
\usepackage[authoryear,round]{natbib}
\let\cite\citep

\usepackage{amsmath,amssymb,amsfonts}
\DeclareOldFontCommand{\cal}{\relax}{\mathcal}

\usepackage{graphicx}
\usepackage{textcomp}
\usepackage{xcolor}
\usepackage{float}
\definecolor{darkgreen}{rgb}{0,0.3,0}
\usepackage[colorlinks,linkcolor=blue,citecolor=darkgreen]{hyperref}
\usepackage{bookmark}

\usepackage[utf8]{inputenc}
\usepackage{amsmath}
\usepackage{graphicx}
\usepackage{amssymb}
\usepackage{verbatim}

\usepackage{enumerate}

\usepackage{exscale}

\usepackage{float}
\usepackage{tikz-cd, circuitikz}
\tikzset{global scale/.style={
    scale=#1,
    every node/.append style={scale=#1}
  }
}

\usepackage{booktabs} 

\newtheorem{definition}{Definition}
\newtheorem{assumption}{Assumption}
\newtheorem{lemma}{Lemma}
\newtheorem{remark}{Remark}

\newtheorem{proposition}{Proposition}

\newtheorem{problem}{Problem}
\newtheorem{condition}{Condition}

\def\begquo{\begin{quote}}
\def\endquo{\end{quote}}
\def\begequarr{\begin{eqnarray}}
\def\endequarr{\end{eqnarray}}
\def\begequarrs{\begin{eqnarray*}}
\def\endequarrs{\end{eqnarray*}}
\def\begarr{\begin{array}}
\def\endarr{\end{array}}
\def\begequ{\begin{equation}}
\def\endequ{\end{equation}}
\def\lab{\label}
\def\begdes{\begin{description}}
\def\enddes{\end{description}}
\def\begenu{\begin{enumerate}}
\def\begite{\begin{itemize}}
\def\endite{\end{itemize}}
\def\endenu{\end{enumerate}}

\def\lef[{\left[\begin{array}}
\def\rig]{\end{array}\right]}
\def\begcen{\begin{center}}
\def\endcen{\end{center}}
\def\begrem{\begin{remark}\rm}
\def\endrem{\end{remark}}
\def\begdef{\begin{definition}}
\def\enddef{\end{definition}}
\def\begpro{\begin{proposition}}
\def\endpro{\end{proposition}}
\def\begfac{\begin{fact}}
\def\endfac{\end{fact}}
\def\begass{\begin{assumption}}
\def\endass{\end{assumption}}
\def\begsubequ{\begin{subequations}}
\def\endsubequ{\end{subequations}}

\def\begmat#1{\begin{bmatrix}#1\end{bmatrix}}

\def\begalis#1{\begin{align*}{#1}\end{align*}}

\def\cale{{\cal E}}
\def\cali{{\cal I}}
\def\calo{{\cal O}}
\def\calh{{\cal H}}

\def\cale{{\cal E}}

\def\calx{\mathcal{X}}

\def\phil{\phi^{\mathtt{L}}}

\def\L2e{{\cal L}_{2e}}

\def\rea{\mathbb{R}}

\def\diag{\mbox{diag}}

\def\col{\mbox{col}}

\def\diag{\mbox{diag}}
\def\rank{\mbox{rank}\;}

\def\min{{\mbox{min}}}
\def\max{{\mbox{max}}}
\def\clo{{\rm cl}}

\def\TAC{{\it IEEE Trans. Autom. Control}}

\def\IJC{{\it Int. J. Control}}
\def\CDC{{\it IEEE Conf. Decis. Control}}
\def\SCL{{\it Syst. Control Lett.}}
\def\AUT{{\it Automatica}}

\usepackage{xcolor}

\DeclareRobustCommand{\cdcchange}[1]{#1}
\def\QED{\hfill $\square$}
\def\qed{\hfill $\triangleleft$}

\usepackage{tikz,xcolor}

\usetikzlibrary{decorations.markings}

\definecolor{lime}{HTML}{A6CE39}
\DeclareRobustCommand{\orcidicon}{
	\begin{tikzpicture}
	\draw[lime, fill=lime] (0,0) 
	circle [radius=0.16] 
	node[white] {{\fontfamily{qag}\selectfont \tiny ID}};
	\draw[white, fill=white] (-0.0625,0.095) 
	circle [radius=0.007];
	\end{tikzpicture}
	\hspace{-2mm}
}

\usepackage{pgfplots}
\pgfplotsset{compat=1.18}

\foreach \x in {A, ..., Z}{\expandafter\xdef\csname orcid\x\endcsname{\noexpand\href{https://orcid.org/\csname orcidauthor\x\endcsname}
			{\noexpand\orcidicon}}
}

\def\BibTeX{{\rm B\kern-.05em{\sc i\kern-.025em b}\kern-.08em
    T\kern-.1667em\lower.7ex\hbox{E}\kern-.125emX}}
\newsavebox{\mathfitbox}
\newcommand{\fitmath}[1]{\sbox{\mathfitbox}{#1}\ifdim\wd\mathfitbox>0.98\linewidth\resizebox{0.98\linewidth}{!}{\usebox{\mathfitbox}}\else\usebox{\mathfitbox}\fi}
\begin{document}

\title{On the Solvability of Parameter Estimation-Based Observers for Nonlinear Systems\thanks{An abridged version of the paper will be presented at the 2026 IEEE Conference on Decision and Control. Corresponding author: Bowen Yi.}}
\author[1]{Bowen Yi}
\author[2]{Leyan Fang}
\author[2]{Romeo Ortega}
\affil[1]{Department of Electrical Engineering, Polytechnique Montr\'eal and GERAD, Montr\'eal, QC H3T 0A3, Canada}
\affil[2]{Department of Electrical and Electronic Engineering, ITAM, Mexico City, 01080 Mexico}
\affil[ ]{\texttt{bowen.yi@polymtl.ca}, \texttt{fangleyan0215@163.com}, \texttt{romeo.ortega@itam.mx}}
\renewcommand{\Affilfont}{\small}
\date{}
\hypersetup{pdftitle={On the Solvability of Parameter Estimation-Based Observers for Nonlinear Systems},pdfauthor={Bowen Yi, Leyan Fang, Romeo Ortega},pdfsubject={},pdfkeywords={nonlinear systems, observer design, PEBO, observability}}
\maketitle

\begin{abstract}
Parameter estimation-based observer (PEBO) is a recently developed constructive tool to design state observers for nonlinear systems. It reformulates the state estimation problem as one of online parameter identification, effectively addressing many open estimation challenges in practical applications. The feasibility of a PEBO design relies on two fundamental properties: transformability and identifiability. The former pertains to the existence of an injective solution to a suitable partial differential equation, whereas the latter characterizes the uniqueness of the parameterization induced by the resulting regression model. In this paper, we analyze the existence of PEBOs for general nonlinear systems by studying these two properties in detail and by providing sufficient conditions under which they hold.
\end{abstract}

\noindent\textbf{Keywords:}
nonlinear system, observer design, parameter estimation-based observer (PEBO), observability

\section{Introduction}
State estimation is a fundamental problem across almost all engineering domains. 
It aims at reconstructing unknown variables of interest by combining \emph{a priori} information---typically in the form of mathematical models and observation functions---with \emph{a posteriori} measurements obtained from sensors. 
In control engineering, particular attention is devoted to the estimation of internal states of dynamical systems described by state-space differential equations using partial and noisy observations.

State estimation algorithms are commonly referred to as state observers. In the deterministic setting, existing observer designs can be broadly classified into two main categories: (a) optimization-based approaches and (b) recursive, or filtering-based, approaches \cite{BER,BES07, HASRAW}. Optimization-based approaches formulate the estimation problem as the solution of a constrained optimization problem over a finite or infinite time horizon, typically by fitting the measured input--output data to a dynamical model. In this framework, the unknown state trajectory is often collected into an extended decision variable, {\em e.g.}, $\mathbf{x} = [x(t_1), \ldots, x(t_q)] \in \rea^{n \times q}$, where $x(t) \in \rea^n$ is the systems state vector.  Representative examples include full-information estimation (FIE) \cite{ALLRAW} and moving-horizon estimation (MHE) \cite{SCHetal23}.  These methods are highly flexible and have become particularly popular in robotics \cite{BAR24}; however, they usually entail a significant computational burden, which may limit their applicability in real-time and embedded settings.

In contrast, recursive or filtering-based approaches rely on the design of a dynamical system that generates state estimates in real time by continuously processing incoming measurements. This class includes the pioneering work of Luenberger \cite{LUE}, Kalman and Bucy filter \cite{KAL}, high-gain observers \cite{KHA07}, Kazantzis--Kravaris--Luenberger (KKL) observers \cite{KAZKRA98,ANDPRA}, and immersion and invariance (I\&I) observers \cite{ASTKARORTbook,KARetal}. See \cite{BERetal22} for a comprehensive review of nonlinear observer designs. These methods are computationally efficient and naturally suited for online implementation. 
However, their applicability typically hinges on stringent structural conditions, which are often difficult to verify for general nonlinear systems.

Recently, the parameter estimation-based observer (PEBO) framework has been proposed as a novel alternative that \emph{bridges these two paradigms} \cite{ORTetal15}. Instead of directly estimating the state, PEBO reformulates the state observation problem as that of identifying an unknown constant parameter through a suitable dynamic extension. As a result, PEBO retains a recursive observer structure by introducing a dynamic extension, while ultimately relying on the solution of a parameter identification or optimization problem. This distinctive feature enables the systematic incorporation of tools from adaptive control and system identification into the observer design, thereby inheriting some of the flexibility of optimization-based approaches without sacrificing the real-time nature of recursive observers.

A significant generalization of PEBO was reported in \cite{ORTetal21}, introducing the so-called generalized PEBO (GPEBO), which extended the methodology from systems with measurable derivative of the state to the more general {\em state-affine} systems \cite{BER}. The reader is referred to  \cite{ORTetal21} and discussion {\bf D1} in Section \ref{sec:5} for additional details on this extension. The PEBO  framework has been extended in several directions, including systems evolving on matrix Lie groups \cite{yi2024pebo, yi2023attitude, yi2022globally}, non-uniformly observable systems \cite{WANetal23}, systems with delayed measurements \cite{BOBetal21} as well as adaptive observers and reduced-order designs \cite{BOBetal22}.  Moreover, PEBO-based techniques have found successful applications in a variety of domains, such as electrical drives \cite{HEetal}, power systems \cite{BOBetal}, and robotic navigation \cite{WENetal25, WULIU}, to name just a few. Despite its conceptual appeal and empirical success, the fundamental question of when---besides state-affine systems---a general nonlinear system admits a PEBO design remains largely open.

In \cite{ortega2024immersion} a heuristic procedure to transform a class of nonlinear systems into state-affine ones---hence, amenable for a GPEBO design---is proposed. For general nonlinear systems the feasibility of PEBO relies on two key properties: \emph{transformability}, which concerns the existence of a suitable invertible coordinate transformation---equivalently, the solvability of a class of partial differential equations (PDEs) admitting injective solutions over time---and \emph{identifiability}, which pertains to the uniqueness of the induced regression model that is used for the identification of the unknown parameters.  At present, these properties are typically verified on a case-by-case basis, and no general solvability results are available. This paper aims to fill this gap by establishing a systematic \emph{existence} result for PEBOs in general nonlinear systems, providing explicit characterizations of both transformability and identifiability conditions.

\emph{Organization.} The remainder of this paper is organized as follows. 
In Section~\ref{sec:2}, we formulate the state estimation problem and briefly review the PEBO framework, highlighting the two key properties of transformability and identifiability. 
In Section~\ref{sec:3}, we study the transformability property and characterize it in terms of the solvability of a class of PDEs, for which constructive solutions are provided. 
Section~\ref{sec:4} is devoted to the identifiability analysis of the induced regression model, where sufficient conditions for both global and local identifiability are established. 
Several discussions and extensions are presented in Section~\ref{sec:5}. 
An illustrative example is given in Section~\ref{sec:6} to demonstrate the proposed results. 
Finally, concluding remarks are drawn in Section~\ref{sec:7}.

\emph{Notation.} $|v|$ represents the Euclidean norm of the vector $v \in \rea^m$, while $\|x(t)\|_{L_{2}}$ is the $L_2$ norm of the function $x(t) \in \rea^n$. All mappings and functions are assumed to be continuously differentiable, {\em i.e.}, $C^1$. We use $\mathbb{C}$ to represent the complex plane, and $\mathbb{C}_{>0}$ ($\mathbb{C}_{<0}$) for the open right (left) half-plane. For a forward complete system $\dot x = f(x,t)$, let $X(x,t;s)$ denote its unique solution at time $s$, with the initial condition $x$ at time $t$.  ${\bf e}_j$ is the canonical basis vector in the Euclidean space of appropriate dimension. For a matrix $A \in \rea^{n\times n}$, we use $\lambda_i(A)$ to represent all its eigenvalues. The notation $\mathbf{1}_n$ represents the $n$-dimensional column vector whose entries are all equal to one. For a set $\mathcal X\subset\mathbb R^n$, $\operatorname{cl}(\mathcal X)$ denotes its closure. When clear from the context, the arguments of functions and operators are omitted. 

A preliminary version of this article will be presented at the conference \cite{CDC}, containing the main results of Section \ref{sec:3}.
\section{Problem Formulation}
\label{sec:2}
Consider nonlinear time-varying systems of the form
\begin{equation}\label{NLsyst}
\begin{aligned}
\begin{aligned}
        \dot x = f(x,t), \quad 
        y = h(x)
\end{aligned}
\end{aligned}
\end{equation}
with state $x(t)\in \rea^n$, measured output $y(t)\in \rea^p$, and known smooth functions $f$ and $h$. We are interested in the online estimation  of the state $x(t)$.

Note that, for the purpose of state estimation, the above systems also encompass control systems of the form $\dot x = f(x,u(t))$, $y= h(x)$ by viewing the input $u(t)$ as an external signal. 

We make the following assumptions.

\begin{assumption}
\rm \label{assm:1} {\em (Forward completeness and boundedness)}
The system \eqref{NLsyst} is forward complete, and its state $x(t)$ evolves in a bounded open set $\mathcal X\subset\mathbb R^n$ for all $t\ge0$.
\end{assumption}

\begin{assumption}
\rm \label{assm:2} (\emph{Distinguishability})
The system \eqref{NLsyst} is backward distinguishable on $\cal X$ with a finite horizon $t_\star>0$, \emph{i.e.}, for any $t\ge t_\star$ and any pair $(x_a, x_b) \in \mathcal{X}^2$,
\begin{equation*}
\fitmath{$\displaystyle
\begin{aligned}
        Y(x_a, t;s) = Y(x_b, t;s),\; \forall s \in [t-&t_\star ,t] 
\; \implies \;  x_a = x_b,
\end{aligned}
$}
\end{equation*}
where $Y(x,t;s):=h(X(x,t;s))$.
\end{assumption}

\begin{assumption}
\rm \label{assm:back} {\em (Backwards boundedness)}
The solutions to the system \eqref{NLsyst} that are initiated from $\calx$ do not blow-up in finite backward time.
\end{assumption}

\vspace{.1cm}

Distinguishability can be viewed as a specific form of observability for nonlinear systems. Intuitively, it means that by using the measured outputs over a backward time interval, one can uniquely determine the state $x$ at the current moment among all possible states in $\calx$. See \cite{BER} for additional discussion on these assumptions.
\subsection{Background Material on PEBO}
To estimate the state $x(t) \in \rea^n$ in real time, \cite{ORTetal15} proposed the novel approach, referred to as PEBO, which reformulates the problem as the online identification of certain constant parameters---which correspond to the initial condition of the state. Below, we briefly outline the PEBO design via highlighting the two key properties of the given system that are required for its implementation.

\subsubsection{Transformability}
For system \eqref{NLsyst}, a key step of PEBO is to look for a $C^1$ mapping 
$
\phi:(x,t)\mapsto z
$ 
such that after a change of coordinate the dynamics are given by the \emph{canonical form}
\begin{equation}
\label{eq:dot_z}
    \dot z = \beta(y,t)
\end{equation}
for some function $\beta(\cdot,\cdot)$ and $z(t)\in \rea^{n_z}$ with $n_z \ge n$. For convenience, we sometimes use the notation $\phi_t(x)$ for $\phi(x,t)$. 
Recall that, for a fixed $t$, the mapping $\phi_t:\mathcal X\to\mathbb R^{n_z}$ is said to be injective on $\mathcal X$ if
$
\phi_t(x_a)=\phi_t(x_b),\; x_a,x_b\in\mathcal X,
$
implies $x_a=x_b$.
Given a function $\beta$, the solvability of the PDE
    \begin{equation}\label{pde}
      {\partial \phi \over \partial t}(x,t) +  {\partial \phi \over \partial x}(x,t)f(x,t) = \beta(h(x),t)
    \end{equation}
is a {\em necessary and sufficient} condition for transforming the nonlinear dynamics \eqref{NLsyst} into the canonical form \eqref{eq:dot_z}.

Since we are interested in estimating the state in the original $x$-coordinate, we should be able to pull it back from $z$-coordinate, and thus impose the {\em left invertibility} of the mapping $\phi(x,t)$ with respect to $x$. This is equivalent to the existence of its left inverse, {\em i.e.}, we need another mapping $\phil(\cdot,t):  \phi_t(\calx) \to \calx$ satisfying 
\begin{equation*}
    \phil\left( \phi(x,t) ,t \right) = x, \quad \forall x \in \calx
\end{equation*}
whenever $\phi_t$ is injective. Hence, given an estimate $\hat z\in\phi_t(\mathcal X)$ of $z=\phi(x,t)$, the state can be reconstructed as $\hat x= \phil(\hat z,t)$.
\subsubsection{Identifiability}
Noting the time derivative of state $z$ is an \emph{available} signal from measured signals\footnote{Hereafter, a derivative is said to be available if the right-hand side of the corresponding differential equation ({\em e.g.}, the function $\beta(y,t)$) depends solely on measurable signals.}, we design the following dynamic extension 
\begin{equation}
\label{eq:dot_zeta}
    \dot \zeta = \beta(y,t) , \quad \zeta(t_0) = \zeta_0,
\end{equation}
where the observer's initial condition $\zeta_0 \in \rea^{n_z}$ can be freely selected by the user at the initial time $t_0$. 
Obviously, as ${d\over dt}(\zeta - z)=0$, there exists a constant vector $\theta \in \rea^{n_z}$ such that 
\begin{equation}
\label{eq:phi_zeta_theta}
    \phi(x,t) = \zeta +\theta, \quad t\ge t_0,
\end{equation}
which translates the estimation of $z(t)$ into the one of the parameter $\theta$---hence, the qualifier PEBO. 

To carry out the aforementioned parameter estimation objective, it is necessary to dispose of a {\em regression equation}, namely a relation of the form
\begin{equation}
\label{eq:regressor}
\calh(t)=\Phi(\theta, t),
\end{equation}
where both $\calh(t)$ and $\Phi(\cdot, t)$ are known mappings constructed from the available signals 
$y$, $\zeta$, and time $t$.

\begin{remark}
\rm 
For example, by substituting $x= \phil(\zeta+\theta ,t)$ into the output function $h(x)$, we have the nonlinearly parameterized regressor 
$
y = h(\phil(\zeta+\theta,t)).
$
The simplest case is the well-known linear regression equation given as
$
y(t)=m^\top(t) \theta,
$
with some $m(t) \in \rea^{n_z}$. Such scenarios are frequently encountered in the modeling of electrical and mechanical systems \cite{ORTetal15}. \hfill $\triangleleft$
\end{remark}

The last step is to find a consistent estimate of $\theta$, denoted as $\hat\theta \in \rea^{n_z}$. It can be obtained via solving the optimization problem
\begin{equation}
\label{eq:parameter_opt}
\underset{\hat\theta \in \rea^{n_z}}{\min} ~ J(\calh(t), \Phi(\hat\theta,t))
\end{equation}
where the cost function $J(\cdot,\cdot)$ is selected to achieve its minimum at $\hat\theta = \theta$. For example, a simple choice is to formulate it as a least squares problem, i.e., $J=\|\calh(t) - \Phi(\hat\theta(t), t)\|^2_{L_{2}}$. If this parameter identification problem is solvable and yields a consistent estimate $\hat\theta$, the estimate of $x$ is then given by 
\begin{equation}
\label{eq:hat_x}
    \hat x(t) = \phil(\zeta(t) + \hat\theta(t),t).
\end{equation}
Here, $\hat z := \zeta + \hat \theta$ can be viewed as an estimate of $z$ in the transformed coordinate.

\begin{remark}
\rm
In many cases, we prefer to design a recursive parameter estimator to solve the optimization problem \eqref{eq:parameter_opt} such that $\hat \theta(t) \to \theta$ as $t\to+\infty$. This may improve robustness with respect to unavoidable measurement noise. Typical examples include gradient descent and recursive least-squares estimators \cite{SASBOD}; see \cite{ORTetal2020ARC} for a survey of recent parameter estimation techniques. Then, if the left inverse function $\phil$ is continuous with respect to $z$, $\hat x$ in \eqref{eq:hat_x} is also an asymptotically convergent estimate.   
\hfill $\triangleleft$
\end{remark}

In Fig. \ref{fig:pebo_diagram}, we briefly illustrate the signal flow in the PEBO approach for state estimation of nonlinear systems.

\begin{figure}[!htp]
\centering
\resizebox{.65\textwidth}{!}{%
\begin{circuitikz}
\tikzstyle{every node}=[font=\LARGE]

\draw[fill=blue!10!white, draw=blue!80!black]
    (5.8,16.75) rectangle (10.2,15.25);

\draw[fill=orange!15!white, draw=orange!70!black]
    (5.8,14) rectangle (10.2,12.5);

\draw[fill=teal!10!white, draw=teal!60!black]
    (11.5,16.75) rectangle (17.25,15.25);

\draw[->, >=Stealth] (3,16) -- (5.8,16);
\draw[->, >=Stealth] (10.2,16) -- (11.5,16);

\draw[->, >=Stealth] (8,15.25) -- (8,14);

\draw[short] (4.5,16) -- (4.5,13.25);
\draw[->, >=Stealth] (4.5,13.25) -- (5.8,13.25);

\draw[short] (10.2,13.25) -- (14.375,13.25);
\draw[->, >=Stealth] (14.375,13.25) -- (14.375,15.25);

\draw[->, >=Stealth] (17.25,16) -- (18.5,16);

\node at (8.45,14.75) {$\zeta$};
\node at (10.85,16.5) {$\zeta$};
\node at (12.4,13.75) {$\hat{\theta}$};
\node at (17.9,16.5) {$\hat{x}$};

\node at (8,16)
    {$\dot{\zeta} = \beta(y,t)$};

\node at (8,13.25)
    {$\hat{\theta} \in \underset{\theta}{\arg\min}\, J$};

\node at (4,16.5)
    {$y$};

\node at (14.375,16)
    {$\hat{x} = \phil(\zeta + \hat{\theta},t)$};

\end{circuitikz}
}%
\caption{Signal flow diagram of the PEBO}
\label{fig:pebo_diagram}
\end{figure}
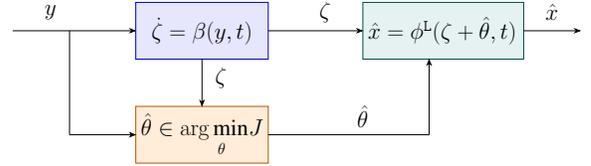
\subsection{Summary of the problem setting}
\label{sec:2b}
From the above discussion, given a nonlinear system of the form \eqref{NLsyst}, the feasibility of designing a PEBO relies on the following two key conditions.
\begin{itemize}
    \item[\bf C1] \emph{Transformability}: The PDE \eqref{pde} admits a solution $\phi$ such that $\phi(\cdot,t)$ is injective on $\calx$ for all $t>t_\phi$, for some finite $t_\phi\ge 0$ with a suitably chosen mapping $\beta$.

    \vspace{.2cm}

    \item[\bf C2] \emph{Identifiability}: It is possible to obtain a regression model of the form \eqref{eq:regressor} that is identifiable with respect to $\theta$.
\end{itemize}

It is natural to ask for which class of nonlinear systems in the form of \eqref{NLsyst} the above conditions are satisfied. In this paper, we address this question by considering the following problem.

\begin{problem}
\rm 
For the given nonlinear system \eqref{NLsyst} under Assumptions \ref{assm:1}-\ref{assm:back}, we seek to determine the conditions under which {\bf C1} and {\bf C2} are feasible.
\end{problem}
\section{Transformability to Canonical Form}
\label{sec:3}
In this section, we study the conditions that allow the nonlinear system \eqref{NLsyst} to be transformed into the canonical form \eqref{eq:dot_z}. This transformation is intrinsically linked to the solvability of the PDE \eqref{pde}, for which the solution $\phi(\cdot, t)$ possesses a left inverse that holds over time.
\subsection{PDE Solvability}
The local solvability of the PDE \eqref{pde} for $\phi$ is discussed in \cite[Section 2.5]{ORTetal15}. In particular, a local construction can be obtained by selecting the $i$-th element $\beta_i$ such that, around some point $x_\star \in \calx$,
\begin{equation}
    \rank\begmat{f(x_\star,t) \\ \beta_i \big(h(x_\star),t \big)} =1,\quad  \forall t\ge 0
\end{equation}
for $i=1, \ldots, n_z$, uniformly over time.

Under the above condition, the Frobenius theorem \cite[Ch 8.3]{SASbook} guarantees the existence of a local solution $\phi(\cdot,\cdot)$ to the PDE \eqref{pde} in a neighborhood of the point $x_\star$. This construction can be carried out elementwise, as detailed in \cite{CHEetal}. However, such a solution is inherently local and does not, in general, provide an explicit construction on the entire set $\calx$. Moreover, it does not ensure the injectivity of $\phi(\cdot,t)$, a key property for constructing the regression model \eqref{eq:regressor} and reconstructing the state from the transformed coordinates. 
To address these limitations, we provide below a constructive solution to \eqref{pde} on $\calx \times [0,+\infty)$ by imposing a particular structure on the mapping $\beta$---motivated by the construction of the KKL observers in \cite{ANDPRA} .

\begin{proposition}
\rm
\label{prop:1}
Consider the system \eqref{NLsyst} with the $C^1$ vector field $f$ and the output function $h$
satisfying Assumptions \ref{assm:1} and \ref{assm:back}. For any Hurwitz
matrix $A \in \rea^{n_z \times n_z}$, any matrix
$B \in \rea^{n_z\times p}$, and any $C^1$ function
$H:\rea^p \times \rea_{\ge 0} \to \rea^p$, the PDE \eqref{pde}
admits a $C^1$ solution
\begin{equation}
\label{eq:phi_solution}
\phi(x,t)
=
\int_0^t
e^{-As}B
H\big(h(X(x,t;s)),s\big)\,ds
\end{equation}
on $\mathcal{X}\times[0,+\infty)$, by choosing
$$
\beta(h(x),t)=e^{-At}BH(h(x),t).
$$
Moreover, choosing
$
H(h(x),t)=e^{-\rho t}h(x),
$
with
$
\rho>-\min_i\{\operatorname{Re}(\lambda_i(A))\},
$
ensures that the dynamic extension $\zeta(t)$ in \eqref{eq:dot_zeta} remains bounded for all $t\ge t_0$.
\end{proposition}

\begin{proof}
Recall that $X(x,t;s)$ denotes the solution at time $s$, initialized from the value $x$ at time $t$. Under Assumption~\ref{assm:back}, $X(x,t;s)$ is well defined for all $s\in[0,t]$. We consider the alternative output
\(
y_a := H(h(x),t),
\)
and define its evolution as
\[
Y_a(x,t;s):=H(h(X(x,t;s)),s).
\]
Since $f$, $h$, and $H$ are $C^1$, the smooth dependence of solutions on their initial conditions implies that the function $\phi$ in \eqref{eq:phi_solution} is $C^1$.

Consider any solution $x(\cdot)$ of \eqref{NLsyst}. By the flow property,
\[
X(x(t),t;s)=x(s), \qquad 0\le s\le t,
\]
and hence
\[
Y_a(x(t),t;s)=H(h(x(s)),s)=y_a(s).
\]
It follows from \eqref{eq:phi_solution} that
$
\phi(x(t),t)
=
\int_0^t e^{-As}B y_a(s)\,ds.
$
Differentiating both sides with respect to $t$ gives
\[
\frac{d}{dt}\phi(x(t),t)
=
e^{-At}B y_a(t)
=
e^{-At}BH(h(x(t)),t).
\]
On the other hand, by the chain rule,
\[
\frac{d}{dt}\phi(x(t),t)
=
{\partial \phi \over \partial t}(x(t),t)
+
{\partial \phi \over \partial x}(x(t),t)f(x(t),t).
\]
Since the trajectory is arbitrary, we obtain
\[
\begin{aligned}
    {\partial \phi \over \partial t}(x,t)
+
{\partial \phi \over \partial x}(x,t)f(x,t)
& =
e^{-At}BH(h(x),t)
\\
& =
\beta(h(x),t),
\end{aligned}
\]
which verifies that $\phi$ in \eqref{eq:phi_solution} is a $C^1$ solution to the PDE \eqref{pde}.

We next consider the particular choice
$H(h(x),t)=e^{-\rho t}h(x)$ to establish the boundedness of $\zeta$.
The dynamic extension \eqref{eq:dot_zeta} satisfies
\begin{equation}
\begin{aligned}
    \zeta(t)
    &=
    \zeta_0+\int_{t_0}^t \beta(y(\tau),\tau)d\tau
    \\
    &=
    \zeta_0+\int_{t_0}^t
    \exp(-A\tau)B H(h(x(\tau)),\tau)d\tau
    \\
    &=
    \zeta_0+\int_{t_0}^t
    \exp(-(A+\rho I)\tau)Bh(x(\tau))d\tau.
\end{aligned}
\end{equation}
Since
$\rho>-\min\{\mbox{Re}(\lambda_i(A))\}$,
the matrix $-(A+\rho I)$ is Hurwitz. Hence, there exist constants
$c,\alpha>0$ such that
\(
\|e^{-(A+\rho I)t}\|
\le c e^{-\alpha t},
\; t\ge0.
\)
Moreover, Assumption~\ref{assm:1} and the continuity of $h$ imply that
$|h(x(t))|\le\bar h$ for some $\bar h>0$. Therefore,
\[
|\zeta(t)-\zeta_0|
\le
c\|B\|\bar h
\int_{t_0}^{t}e^{-\alpha\tau}d\tau
\le
\frac{c\|B\|\bar h}{\alpha}e^{-\alpha t_0},
\]
which proves that $\zeta(t)$ remains bounded for all $t\ge t_0$.
\QED
\end{proof}

\begin{remark}\rm 
In terms of the above proposition, for any $C^1$ function $H(y,t)$,
we are always able to construct a function $\phi(x,t)$ such that the
corresponding dynamic extension
\(
\dot\zeta=e^{-At}BH(y,t)
\)
yields the invariant foliation
$$
    \{{(x, \zeta)\in \rea^n \times \rea^{n_z}}: \phi(x,t) = \zeta + \theta, \; \forall t \ge t_0\}.
$$
See Fig. \ref{fig:invariant} for a geometric illustration.
\hfill $\triangleleft$
\end{remark}

\begin{figure}[!htp]
    \centering

\begin{tikzpicture}[>=stealth,scale=1.2]
\draw[->] (0,0,0) -- (2,0,0) node[below] {};
\draw[->] (0,0,0) -- (0,1,0) node[left] {};
\draw[->] (0,0,0) -- (0,0,4) node[above] {};
\fill[blue!20,opacity=0.6] (0,0,1.5) to[out=0,in=180] (1.5,0,2.5) 
to[out=90,in=270] (1.5,1.5,2.8) to[out=180,in=0] (0,1.5,1.8) to[out=270,in=90] (0,0,1.5);
\draw[thick] (0,0,1.5) to[out=0,in=180] (1.5,0,2.5) to[out=90,in=270] (1.5,1.5,2.8)
to[out=180,in=0] (0,1.5,1.8) to[out=270,in=90] (0,0,1.5);
\node at (2.8,0.6,2.8) {$\phi(x,t) = \zeta(t)+\theta$};
  \draw[thick, blue, postaction={decorate},
decoration={markings, mark=at position 0.5 with {\arrow{stealth}}}] 
    (0.5,0,1.7) to[out=35,in=215] (0.8,0.3,2.0) 
    to[out=25,in=205] (1.0,0.6,2.2) 
    to[out=350,in=190] (1.2,0.8,2.4) 
    to[out=15,in=185] (1.0,1.4,2.6);
\fill[red!20,opacity=0.6] (0,0,3) to[out=0,in=180] (1.5,0,4) 
to[out=90,in=270] (1.5,1.5,4.3) to[out=180,in=0] (0,1.5,3.3) to[out=270,in=90] (0,0,3);
\draw[thick,dashed] (0,0,3) to[out=0,in=180] (1.5,0,4) to[out=90,in=270] (1.5,1.5,4.3)
to[out=180,in=0] (0,1.5,3.3) to[out=270,in=90] (0,0,3);
\node at (1.3,.6,4.5) {$\zeta(t)$};
\draw[thick, red, postaction={decorate},
decoration={markings, mark=at position 0.4 with {\arrow{stealth}}}] 
    (0.3,0,3.1) to[out=35,in=215] (0.7,0.3,3.4) 
    to[out=25,in=205] (0.9,0.6,3.6) 
    to[out=350,in=190] (1.1,0.8,3.8) 
    to[out=15,in=185] (.9,1.4,4.0);
\end{tikzpicture}
    \caption{Invariant foliation in PEBO 
    }
    \label{fig:invariant}
\end{figure}
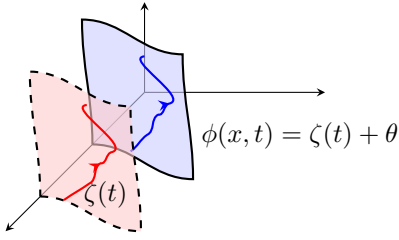

\begin{remark}\rm
The choice
\(
H(h(x),t)=e^{-\rho t}h(x)
\)
ensures boundedness of the dynamic extension, but progressively attenuates the contribution of newly arriving measurements, thereby reducing sensitivity to recent information over long time horizons. This issue can be mitigated by periodically re-initializing the dynamic extension and implementing PEBO over moving horizons \cite{yi2024imu}.
\hfill $\triangleleft$
\end{remark}

\subsection{Left Invertibility}
In this subsection, we show that a suitable choice of the controllable pair $(A,B)$ and the function $H(\cdot,t)$ can ensure that the mapping $\phi(x,t)$ is {\em injective} on $\calx$, thereby allowing the construction of a left inverse $\phil$ such that $\phil(\phi(x,t),t) =x$ for all $x\in \calx$.

We impose the following condition on the function $H$.

\begin{condition}
\rm \label{cond:1}
For each $t\in[0,\infty)$, the map
\(
H_t := H(\cdot,t):\mathbb{R}^p \to \mathbb{R}^p
\)
is a $C^1$ diffeomorphism.
\hfill $\triangleleft$
\end{condition}

Under Condition~\ref{cond:1}, the mapping $\phi(\cdot,t)$ constructed in Proposition~\ref{prop:1} can be shown to be injective, as stated below.

\begin{proposition}
\rm\label{prop:2}
Consider the complex-valued extension of Proposition~\ref{prop:1} with
\begin{equation}
\label{eq:AB}
\begin{aligned}
A & = I_p \otimes \Lambda
\\
B & = I_p \otimes \mathbf{1}_{n+1}
\\
\Lambda & = \diag(\lambda_1, \ldots, \lambda_{n+1}).
\end{aligned}
\end{equation}
Under {\em Assumptions}~\ref{assm:1}--\ref{assm:back}, with $H$ satisfying
{\em Condition}~\ref{cond:1}, there exists a set
$\cale \subset \mathbb{C}^{n+1}$ of zero Lebesgue measure such that, for any
$(\lambda_1, \ldots, \lambda_{n+1})
\in \mathbb{C}^{n+1}_{<0}\setminus \cale$,
the mapping $\phi(\cdot,t)$ defined in \eqref{eq:phi_solution}
is injective on $\calx$ for all $t>t_\star$.
\end{proposition}
\begin{proof}
Define
\(
T(x,t):=e^{At}\phi(x,t).
\)
It follows from \eqref{eq:phi_solution} that
\[
T(x,t)
=
\int_0^t e^{A(t-s)}B Y_a(x,t;s)\,ds.
\]
Since $e^{At}$ is nonsingular for any finite time
$t\in[0,\infty)$, the injectivity of $\phi(\cdot,t)$ is equivalent
to that of $T(\cdot,t)$.

According to {\em Condition}~\ref{cond:1}, $H(\cdot,t)$ is a
diffeomorphism and thus the system with the available output
\begin{equation*}
    \dot x = f(x,t), \quad y_a = H(h(x),t)
\end{equation*}
inherits the backward distinguishability of the given plant
\eqref{NLsyst}. To see this, we have
\begin{equation*}
\fitmath{$\displaystyle
\begin{aligned}
    &H(h(X(x_a,t;s)),s)
    =
    H(h(X(x_b,t;s)),s),
    \quad \forall s\in[t-t_\star,t]
    \\
    &\iff
    h(X(x_a,t;s))
    =
    h(X(x_b,t;s)),
    \quad \forall s\in[t-t_\star,t]
    \\
    &\implies x_a=x_b.
\end{aligned}
$}
\end{equation*}
Hence, the system with the output $y_a$ is also
backward-distinguishable. Together with Assumption~\ref{assm:back}, this ensures that the hypotheses required in \cite[Theorem~3]{BERAND} are satisfied. Up to a fixed permutation of coordinates, $T$ has the form considered therein. Therefore, there
exists a set $\cale\subset\mathbb C^{n+1}$ of zero Lebesgue measure
such that, for any
$(\lambda_1,\ldots,\lambda_{n+1})
\in\mathbb C_{<0}^{n+1}\setminus\cale$,
$T(\cdot,t)$---and equivalently $\phi(\cdot,t)$---is injective on
$\calx$ for all $t>t_\star$. This completes the proof.
\QED
\end{proof}

The injectivity of $\phi(\cdot,t)$ guarantees the existence of a left-inverse mapping
$\phil(\cdot,t):\phi(\calx,t)\to\calx$ satisfying $\phil(\phi(x,t),t) =x$, $\forall x\in \calx,\; t >t_\star$.
Using the extended variable $\zeta$ defined in \eqref{eq:dot_zeta},
we naturally obtain the regression model
\begin{equation}
\label{eq:NLR}
y = h\circ \phil(\zeta+\theta,t),
\end{equation}
which is well defined for
$t>\max\{t_0,t_\star\}$.

The remaining question is whether the parameter $\theta$ can be uniquely reconstructed from the above regressor using online data, which will be addressed in the next section.

\begin{remark}\rm
The construction in Proposition~\ref{prop:1} also reveals a direct connection between the KKL and PEBO transformations. Indeed, define
\(
T(x,t) =e^{At}\phi(x,t).
\)
The mapping $T$ is precisely
the time-varying KKL transformation associated with the alternative
output $y_a=H(h(x),t)$, yielding the dynamics
\(
\dot\xi=A\xi+By_a.
\)
The PEBO coordinate is then obtained through the time-varying linear
transformation
\(
z=\psi(\xi,t):=e^{-At}\xi,
\)
so that
\(
\phi(x,t)=\psi(T(x,t),t).
\)
Hence, the relation between the two transformations can be summarized
by the commutative diagram.
\[
\begin{tikzcd}
x \arrow[rr, "{\phi(x,t)}"] \arrow[dr, "{T(x,t)}"'] & & z \\
& \xi \arrow[ur, "{\psi(\xi,t)}"']
\end{tikzcd}
\]
Since $\psi(\cdot,t)$ is a diffeomorphism for every finite $t\ge0$,
$T(\cdot,t)$ and $\phi(\cdot,t)$ have the same injectivity property.
\hfill $\triangleleft$
\end{remark}

\begin{remark}\rm
The above proposition is stated over the complex field, but each complex coordinate can be equivalently represented by its real and imaginary parts, yielding a real-valued realization of twice the dimension; see, {\em e.g.}, \cite{ANDPRA}.
\hfill $\triangleleft$
\end{remark}
\section{Parameter Identifiability}
\label{sec:4}
In this section, we address the fundamental question {\bf C2}, the identifiability of the parameter $\theta$ in the resulting regression model \eqref{eq:NLR}.

To carry out the subsequent analysis, we introduce the following definition of identifiability of a regressor, which is adapted from \cite[Definition 3.1]{GEVetal} of a {\em model structure}.

\begin{definition}
\rm (\emph{Identifiability})
Given a regressor 
$$
\calh(t) = \Phi(\theta,t)
$$ 
where the signal $\calh(t)$ and the smooth function 
$\Phi: \rea^{n_\theta}\times [0,+\infty) \to \rea^m$ are known and $\theta$ is in the parameter set $\Theta \subseteq \rea^{n_\theta}$. The regressor is said to be locally identifiable at a value 
$\theta_1 \in \Theta$ over the interval $[t_1, t_2]$ 
(with $t_2>t_1$) if there exists $\delta >0$ such that, for all 
$\theta \in B(\theta_1,\delta)$, where
$$
B(\theta_1, \delta):= \{ z \in \Theta: |z - \theta_1| < \delta \} \subset \Theta,
$$
the following implication is true
\begin{equation}
    \Phi(\theta,t) = \Phi(\theta_1, t),~~\forall t \in [t_1,t_2] 
    \; \implies \theta_1 = \theta.
\end{equation}
If the above implication holds for all $\theta \in \Theta$, 
we call it globally identifiable on $\Theta$ over $[t_1,t_2]$.
\hfill $\triangleleft$
\end{definition}

\subsection{Global Identifiability Under Differential Observability}
\label{sec:4a}
To continue our analysis of identifiability, we assume that the functions $f(x,t)$ and $h(x)$ are sufficiently smooth in their arguments such that the derivatives introduced below are well defined. For convenience, given the system \eqref{NLsyst} we define the $k$-th order observability map 
\begin{equation}
    \calo_k (x,t):= \begmat{h(x) \\ 
   D_{f} h(x,t)
    \\
    \vdots
    \\
    D_f^kh(x,t)
    },
\end{equation}
where the differential operator $D_f$ is defined as
$$
D_f(\cdot):= L_{f}(\cdot) + {\partial \over \partial t}(\cdot)
$$
and $L_f(\cdot)$ denotes the standard Lie derivative along the vector field $f$.

We are now in a position to present the main result on the identifiability of the regressor \eqref{eq:regressor} arising in PEBO.

\begin{proposition}\rm 
\label{prop:global_id}
Consider Propositions \ref{prop:1}--\ref{prop:2}
under {\em Assumptions} \ref{assm:1}--\ref{assm:back}, 
{\em Condition} \ref{cond:1}, and with $(A,B)$ selected as 
\eqref{eq:AB}. 
Suppose that there exist
$k\in\mathbb{N}_+$ and an instant
$t_c>\max\{t_0,t_\star\}$ such that
$\mathcal O_k(\cdot,t_c)$ is an {\em injective immersion} on
$\operatorname{cl}(\mathcal X)$.
Then, for almost all choices of
$(\lambda_1,\ldots,\lambda_{n+1})
\in\mathbb{C}_{<0}^{n+1}$ and for every $t>t_c$,
the regressor \eqref{eq:NLR} is globally identifiable on
\(
   \Theta(t):=\bigcap_{\tau\in[t_c,t]}\Omega_\tau
\)
over the interval $[t_c,t]$, where
\[
    \Omega_\tau
    :=
    \left\{
    \vartheta\in\rea^{n_z}:
    \vartheta+\zeta(\tau)\in\phi(\calx,\tau)
    \right\}.
\]
Moreover, the optimization problem
\begin{equation}
\label{eq:opt}
    \hat\theta
    \in
    \underset{\hat\theta\in\Theta(t)}{\arg\min}
    ~J(\hat\theta),
\end{equation}
with
\begin{equation}
\label{eq:J}
    J(\hat\theta)
    =
    \int_{t_c}^{t}
    \left|
    y(\tau)
    -
    h\circ\phil\big(
    \zeta(\tau)+\hat\theta,\tau
    \big)
    \right|^2 d\tau,
\end{equation}
admits the true parameter vector $\theta$ as its unique global
minimizer for every $t>t_c$.
\qed
\end{proposition}

Its proof can be found in Appendix \ref{sec:app_a}.

\subsection{Remarks}
\lab{subsec42}
Some remarks regarding the identifiability are in order.

\noindent {\bf R1} The high-order time derivatives of the output appear only in the sufficient condition used to establish identifiability in the above proposition. They do not need to be computed or measured when solving the optimization problem \eqref{eq:opt}--\eqref{eq:J}, which involves only the measured output $y$ and the available signals.

\noindent {\bf R2}  The injective immersion assumption on the mapping $\calo_k(\cdot,t_c)$ implies the instantaneous distinguishability of the state $x$ at time $t_c$. Moreover, the mapping $\phi(\cdot, t_c)$ together with the relation $\phi(x(t_c), t_c) = \zeta(t_c) + \theta$ provides the key mechanism for identifying the parameter $\theta$. This highlights the underlying idea behind the identifiability analysis.%

\noindent {\bf R3} In the above result, we only require $\calo_k(\cdot,t_c)$ to be an injective immersion at a single moment $t_c$. This is weaker than the concept of strong differential observability \cite[Definition 6.1]{BERetal22}, since we do not require the injectivity uniformly over time $t$.

\noindent {\bf R4} The above condition for global identifiability is only sufficient. It may be further relaxed for the identification of $\theta$. In particular, even if $\calo_k(\cdot,t)$ is not injective at any single instant over the observation interval, it may still be possible to identify $\theta$ by fusing historical measurements.

\noindent {\bf R5}  The constraint
\(
\hat\theta\in\Theta(t)
:=
\bigcap_{\tau\in[t_c,t]}\Omega_\tau
\)
in \eqref{eq:opt} ensures that
$\zeta(\tau)+\hat\theta\in\phi_\tau(\calx)$ for all
$\tau\in[t_c,t]$, so that the left inverse
$\phi^{\tt L}(\zeta(\tau)+\hat\theta,\tau)$ is well defined.
Notice that the cumulative feasible set $\Theta(t)$ is
nonincreasing with time, namely,
\(
\Theta(t_2)\subseteq\Theta(t_1),
\;
t_2\ge t_1\ge t_c.
\)
See Fig. \ref{fig:twofigs1}.
In practical implementations, the explicit computation of the
sets $\Omega_\tau$ may be avoided if $\phi^{\tt L}$ admits a
suitable extension to the optimization domain. To see this,
define
\(
g(\theta,t):=
\phi^{\tt L}(\zeta(t)+\theta,t).
\)
If there exist instants
$\tau_j\in[t_c,t]$, $j=1,\ldots,\ell$, such that the mapping
\[
\theta\mapsto
\begin{bmatrix}
h(g(\theta,\tau_1))\\
\vdots\\
h(g(\theta,\tau_\ell))
\end{bmatrix}
\]
is injective on the domain of interest, then the corresponding
historical output measurements uniquely determine $\theta$.
Consequently, provided that the extended inverse remains
well defined over the optimization domain, one may solve instead
the unconstrained problem
\(
\min_{\hat\theta\in\mathbb R^{n_z}} J(\hat\theta).
\)

\begin{figure}[!htbp]
    \centering
    \begin{subfigure}[b]{\linewidth}
        \centering
\begin{tikzpicture}[x=0.75pt,y=0.75pt,yscale=-1,xscale=1, global scale = .6]
\draw  [fill={rgb, 255:red, 119; green, 204; blue, 204 }  ,fill opacity=1 ][line width=1.5]  
(62,169.6) .. controls (89,158.6) and (112,155.6) .. (161,184.6) .. controls (210,213.6) and (149,221.6) .. (169,251.6) .. controls (189,281.6) and (68,307.6) .. (49,275.6) .. controls (30,243.6) and (64.88,248.85) .. (68,227.6) .. controls (71.13,206.35) and (35,180.6) .. (62,169.6) -- cycle ;

\draw  [fill={rgb, 255:red, 230; green, 230; blue, 250 }  ,fill opacity=1 ][line width=1.5]  
(258,157.6) .. controls (285,146.6) and (321,160.6) .. (357,172.6) .. controls (393,184.6) and (369,195.6) .. (365,239.6) .. controls (361,283.6) and (339,268.6) .. (293,274.6) .. controls (247,280.6) and (260.88,236.85) .. (264,215.6) .. controls (267.13,194.35) and (231,168.6) .. (258,157.6) -- cycle ;

\draw  [fill={rgb, 255:red, 135; green, 206; blue, 250 },fill opacity=1, line width=1.5, dashed]  
(285,200) .. controls (305,190) and (335,200) .. (340,220) .. controls (345,240) and (320,250) .. (300,245) .. controls (280,240) and (265,220) .. (285,200) -- cycle ;

\draw    (108,191.6) .. controls (147.2,162.2) and (248.83,161.23) .. (288.64,189.81) ;
\draw [shift={(291,191.6)}, rotate = 218.66] 
[fill={rgb, 255:red, 0; green, 0; blue, 0 }][line width=0.08] [draw opacity=0] (8.93,-4.29) -- (0,0) -- (8.93,4.29) -- cycle ;

\draw (86,199.4) node [anchor=north west][inner sep=0.75pt]  [font=\Large]  {$\calx$};
\draw (284,195.4) node [anchor=north west][inner sep=0.75pt]  [font=\Large]  {$\Omega_t
$};

\draw (305,230) node [anchor=north west][inner sep=1pt, font=\Large] {$\theta$};

\draw[fill=red, draw=none]  (305,230)  node[below=0.5em]{} circle(0.2em);

\draw[fill=blue, draw=none]  (340,180)  node[below=0.5em]{} circle(0.2em);

\draw (340,187) node [anchor=north west][inner sep=1pt, font=\Large] {$\theta_a$};
\end{tikzpicture}
        \caption{At a single instant $t$} 
        \label{fig:subfig1}
    \end{subfigure}
\hspace{-0.2cm}
    \begin{subfigure}[b]{\linewidth}
        \centering
        
\begin{tikzpicture}[x=0.75pt,y=0.75pt,yscale=-1,xscale=1, global scale = .6]

\draw  [fill={rgb, 255:red, 119; green, 204; blue, 204 } ,fill opacity=1 ][line width=1.5]  
(62,169.6) .. controls (89,158.6) and (112,155.6) .. (161,184.6) .. controls (210,213.6) and (149,221.6) .. (169,251.6) .. controls (189,281.6) and (68,307.6) .. (49,275.6) .. controls (30,243.6) and (64.88,248.85) .. (68,227.6) .. controls (71.13,206.35) and (35,180.6) .. (62,169.6) -- cycle ;

\draw  [fill={rgb, 255:red, 230; green, 230; blue, 250 }  ,fill opacity=1 ][line width=1.5]  
(258,157.6) .. controls (285,146.6) and (321,160.6) .. (357,172.6) .. controls (393,184.6) and (369,195.6) .. (365,239.6) .. controls (361,283.6) and (339,268.6) .. (293,274.6) .. controls (247,280.6) and (260.88,236.85) .. (264,215.6) .. controls (267.13,194.35) and (231,168.6) .. (258,157.6) -- cycle ;

\draw  [fill=green!40, fill opacity=0.5, dashed, line width=1.2]  
(285,190) .. controls (300,180) and (330,190) .. (340,210) .. controls (345,230) and (320,240) .. (300,235) .. controls (280,230) and (270,210) .. (285,190) -- cycle ;

\draw  [fill=blue!40, fill opacity=0.5, dashed, line width=1.2]  
(295,205) .. controls (315,195) and (335,205) .. (345,225) .. controls (350,245) and (325,255) .. (305,250) .. controls (285,245) and (275,225) .. (295,205) -- cycle ;

\draw  [fill=red!40, fill opacity=0.5, dashed, line width=1.2]  
(275,210) .. controls (295,200) and (325,215) .. (335,235) .. controls (340,255) and (315,260) .. (290,255) .. controls (270,245) and (260,225) .. (275,210) -- cycle ;

\draw    (108,191.6) .. controls (147.2,162.2) and (248.83,161.23) .. (288.64,189.81) ;
\draw [shift={(291,191.6)}, rotate = 218.66] 
[fill={rgb, 255:red, 0; green, 0; blue, 0 }][line width=0.08] [draw opacity=0] (8.93,-4.29) -- (0,0) -- (8.93,4.29) -- cycle ;

\draw[fill=red, draw=none] (305,230) circle (3pt);

\draw (330,190) node [anchor=west, font=\large] {$\Omega_{t_1}$};
\draw (345,225) node [anchor=west, font=\large] {$\Omega_{t_2}$};
\draw (320,255) node [anchor=west, font=\large] {$\Omega_{t_3}$};

\draw (305,230) node [anchor=north west][inner sep=1pt, font=\Large] {$\theta$};

\draw (86,199.4) node [anchor=north west][inner sep=0.75pt]  [font=\Large]  {$\calx$};
\end{tikzpicture}
        \caption{At several instants $\bigcap_{t_i}\Omega_{t_i}$}
        \label{fig:subfig2}
    \end{subfigure}
    \caption{An illustration to \textbf{R5}}
    \label{fig:twofigs1}
\end{figure}

\subsection{Local Identifiability}
\lab{subsec43}

In this section, we establish less restrictive conditions for local, rather than global, identifiability of the regressor
\eqref{eq:NLR}.

Recall the variational system introduced in the proof of
Proposition~\ref{prop:global_id}. Here, we consider its general form along
$X(x,t;s)$ initialized at an arbitrary time $t$:
\begin{equation}
\label{eq:ltv}
\begin{aligned}
    	{d\over ds}\delta x = \underbrace{ {\partial f \over \partial x} (X(x,t;s) ,s) } _{:= A_\delta(s)} \delta x
	,
	\quad 
	\delta y = \underbrace{{\partial h \over \partial x} (X(x,t;s)) } _{C_\delta(s)}
	\delta x.
\end{aligned}
\end{equation}
For an initial variation $\delta x\in\mathbb R^n$ at time $t$,
we denote by $\delta X(x,\delta x,t;s)$ the corresponding solution
of \eqref{eq:ltv}, satisfying
\[
\delta X(x,\delta x,t;t)=\delta x.
\]

\begin{assumption}
\rm \label{assm:4}
(\emph{Infinitesimal distinguishability} \cite[Assumption 3]{AND}) 
For every $t\ge t_\star$ and $x\in\calx$, the variational
system \eqref{eq:ltv} along the trajectory $X(x,t;s)$ is
distinguishable over the backward interval $[t-t_\star,t]$.
That is, for any $\delta x\in\rea^n$,
\begin{equation}
\label{cond:inf_dis}
\begin{aligned}
C_\delta(x,t;s)\delta X(x,\delta x,t;s)=0,
\;\forall s\in[t-t_\star,t]
\;\Longrightarrow\;
\delta x=0,
\end{aligned}
\end{equation}
where
\(
C_\delta(x,t;s)
:=
{\partial h\over\partial x}\big(X(x,t;s)\big).
\)

\end{assumption}

\begin{remark}
\rm 
Note that Assumption \ref{assm:2} imposes  distinguishability on the original nonlinear system \eqref{NLsyst}. In contrast, Assumption \ref{assm:4} requires distinguishability of the associated infinitesimal variational system \eqref{eq:ltv} along the trajectories $X(x,t;s)$, for all $x \in \calx$, generated by \eqref{NLsyst}.
\hfill $\triangleleft$
\end{remark}

\begin{remark}
\rm 
The variational system \eqref{eq:ltv} and infinitesimal distinguishability are fundamentally important in contraction analysis \cite{FORSEP} and contracting observer synthesis \cite{SANPRA, YIetal2021}.
\hfill $\triangleleft$
\end{remark}

In the following proposition, we establish a connection between infinitesimal distinguishability and local identifiability of the regressor arising in PEBO.

\begin{proposition}
\rm \label{prop:id}
Consider Propositions \ref{prop:1}--\ref{prop:2} under
{\em Assumptions} \ref{assm:1}--\ref{assm:4} and
{\em Condition} \ref{cond:1}, with $(A,B)$ given by \eqref{eq:AB}.
Define, along the true trajectory $x(\cdot)$,
$
G(t)
:=
{\partial h \over \partial x}(x(t))
{\partial \phi^{\tt L}\over \partial z}
\big(\phi(x(t),t),t\big).
$
If there exist $t_1>t_\star$ and $T>t_1$ such that
$\phi^{\tt L}$ admits a continuously differentiable extension,
denoted by the same symbol, to an open neighborhood of
$\{\phi(x(t),t):t\in[t_1,T]\}$, and the Gramian
\begin{equation}
\label{eq:W}
W_\Phi[t_1,T]
:=
\int_{t_1}^{T}
G(\tau)^\top G(\tau)\,d\tau
\end{equation}
is positive definite, then there exists a zero-Lebesgue measure
set $\cale_{\tt u}$ such that, for
$(\lambda_1,\ldots,\lambda_{n+1})
\in\mathbb C_{<0}^{n+1}\setminus\cale_{\tt u}$,
the regressor \eqref{eq:NLR} is locally identifiable in $\theta$
over $[t_1,T]$.
\end{proposition}
\begin{proof}
1) Firstly, we show that the mapping $\phi(\cdot,t)$ is an injective
immersion for all $t>t_\star$. According to Proposition
\ref{prop:2}, the mapping $\phi(\cdot,t)$ is injective for all
$t\in(t_\star,\infty)$, over which the left inverse mapping
$\phil(\cdot,t)$ exists, when selecting
$(\lambda_1,\ldots,\lambda_{n+1})
\in\mathbb{C}_{<0}^{n+1}\setminus\cale$.

Under Assumption \ref{assm:4}, the variational system
\eqref{eq:ltv} is distinguishable over every backward interval
$[t-t_\star,t]$. Repeating the argument used in the proof of
Proposition \ref{prop:global_id}, with $t_c$ replaced by an arbitrary
$t>t_\star$, there exists a set $\cale_{\tt v}$ of zero Lebesgue
measure such that, for any
$(\lambda_1,\ldots,\lambda_{n+1})
\in\mathbb{C}_{<0}^{n+1}\setminus\cale_{\tt v}$,
the corresponding variational mapping
$\bar T(\delta x,x,t)$ is injective in $\delta x$ for all
$t>t_\star$. Moreover, as shown in the proof of
Proposition \ref{prop:global_id},
\(
    \bar T(\delta x,x,t)
    =
    {\partial T\over\partial x}(x,t)\delta x.
\)
Hence,
\(
    \operatorname{rank}
    {\partial T\over\partial x}(x,t)
    =n,
    \; t>t_\star.
\)
Since $T(x,t)=e^{At}\phi(x,t)$ and $e^{At}$ is nonsingular,
$\nabla_x\phi(x,t)$ also has full column rank for all
$t>t_\star$. Defining
\(
    \cale_{\tt u}:=\cale_{\tt v}\cup\cale,
\)
we conclude that $\phi(\cdot,t)$ is an injective immersion
for all $t>t_\star$.

2) The second step of the proof is to verify that the Gramian
associated with the regressor \eqref{eq:NLR} is given
by \eqref{eq:W}. Define
\[
    \Phi(\theta,t)
    :=
    h\circ\phil(\zeta(t)+\theta,t).
\]
Since $t_1>t_\star$, the mapping $\phil(\cdot,t)$ is well
defined for all $t\in[t_1,T]$.
By the extension assumption, $\phil$ is continuously
differentiable on an open neighborhood of the relevant trajectory.
Hence, $\Phi(\hat\theta,t)$ is well defined for all $\hat\theta$
sufficiently close to $\theta$, uniformly for $t\in[t_1,T]$.
Moreover, along the true
trajectory,
\(
    \phi(x(t),t)=\zeta(t)+\theta.
\)
Therefore,
\begin{equation}
\begin{aligned}
    {\partial\Phi\over\partial\theta}(\theta,t)
    &=
    {\partial h\over\partial x}(x(t))
    {\partial\phil\over\partial z}
    \big(\phi(x(t),t),t\big)
    \\
    &=G(t).
\end{aligned}
\end{equation}
It follows that
\[
\begin{aligned}
    W_\Phi[t_1,T]
    &:=
    \int_{t_1}^{T}
    \left[
    {\partial\Phi\over\partial\theta}(\theta,\tau)
    \right]^\top
    {\partial\Phi\over\partial\theta}(\theta,\tau)
    d\tau
    \\
    &=
    \int_{t_1}^{T}
    G(\tau)^\top G(\tau)d\tau,
\end{aligned}
\]
which is positive definite by \eqref{eq:W}.

3) The last step is to verify the local identifiability of the
nonlinear regressor \eqref{eq:NLR}, i.e.,
\[
    y(t)=\Phi(\theta,t).
\]
The continuously differentiable extension of
$\phil$, together with the compactness of $[t_1,T]$, ensures that
the following first-order expansion holds uniformly in time.
Similar to \cite[Prop.~3.1]{GEVetal}, for $\hat\theta$
sufficiently close to $\theta$, we have, uniformly for
$t\in[t_1,T]$,
\[
    \Phi(\hat\theta,t)
    =
    \Phi(\theta,t)
    +
    {\partial\Phi\over\partial\theta}(\theta,t)
    [\hat\theta-\theta]
    +
    o\big(|\hat\theta-\theta|\big).
\]
Consequently,
\begin{equation}
\begin{aligned}
    \int_{t_1}^{T}
    \left|
    \Phi(\theta,s)-\Phi(\hat\theta,s)
    \right|^2 ds
    &=
    (\hat\theta-\theta)^\top
    W_\Phi[t_1,T]
    (\hat\theta-\theta)
    \\
    &\quad+
    o\big(|\hat\theta-\theta|^2\big).
\end{aligned}
\end{equation}
Since $W_\Phi[t_1,T]\succ0$, the right-hand side is strictly
positive for every $\hat\theta\neq\theta$ sufficiently close
to $\theta$. Hence, the nonlinear regressor \eqref{eq:NLR}
is locally identifiable in $\theta$ over $[t_1,T]$.
\QED
\end{proof}

\section{Discussion}
\label{sec:5}
In this section we carry out some discussion on the derivations above and some possible extensions.

\noindent {\bf D1} As discussed in the introduction, in \cite{ORTetal21}, the idea of PEBO was extended to the so-called GPEBO, for which the canonical form is given by
\begin{equation}
    \dot z = A(u,y,t) z + \beta(u,y,t).
\end{equation}
By introducing the dynamic extensions 
\begalis{
\dot \zeta &= A(u,y,t)\zeta + \beta(u,y,t)\\
\dot \Omega &= A(u,y,t)\Omega,
} 
with $\Omega(0)=I$, the problem of estimating $z$ is reformulated as that of estimating the constant parameter $\theta := z(0) - \zeta(0)$. Following this research line \cite{yi2024pebo, yi2023attitude, yi2022globally}, the same principle is employed to recast state estimation as a parameter identification problem; however, this is carried out on matrix Lie groups rather than in Euclidean spaces.

\noindent {\bf D2} In the presence of model mismatch and measurement noise, the ``parameter'' vector $\theta$ defined in \eqref{eq:phi_zeta_theta} should in fact be regarded as a (slowly) time-varying signal rather than a constant. In this setting, it is more appropriate to adopt online identification schemes, instead of the batch formulation in \eqref{eq:J}, which may suffer from error accumulation after a long period. However, online identification methods ({\em e.g.}, gradient descent or recursive least squares) typically require stronger conditions, such as persistency of excitation (PE), rather than mere identifiability. See, for instance, \cite{CHOWetal,ORTetal22,WANetal24} for some recent efforts toward relaxing the PE condition for online parameter estimation. In particular it is interesting to recall that in \cite{ORTetal26FCT} it has been shown that---a slight variation---to the standard LS ensures {\em convergence in finite time} with the weakest interval excitation assumption. It is also promising to explore recent advances in online optimization \cite{LES,SCHetal}  within the PEBO framework to handle more complex regressors, such as the nonlinear and non-monotonic structure in \eqref{eq:NLR}.

\noindent {\bf D3} Another approach to mitigate the fragility of PEBO---mainly caused by error accumulation---is to periodically reinitialize the observer every $T_r$ units of time. This idea is closely related to the preintegration approach \cite{LUPSUK, BARBON}. Some interesting and direct connections between preintegration and PEBO have been thoroughly discussed in \cite{yi2024imu}. IMU preintegration consists in summarizing high-rate inertial measurements over a finite interval into relative motion increments, such as attitude, velocity, and position increments, which can be directly used in subsequent state-estimation procedures. It has nowadays become a standard tool in modern robotic navigation and localization.

\noindent {\bf D4}
In the above analysis, under an additive Gaussian measurement-noise model with identity covariance matrix, the Gramian $W_\Phi[t_1,T]$ coincides with the Fisher information matrix \cite[page 217]{LJU}.

\noindent {\bf D5} In \cite{AND}, it is shown that distinguishability and infinitesimal distinguishability guarantee the existence of an exponentially convergent KKL observer for \emph{autonomous} systems. However, its convergence speed is \emph{not tunable}, mainly due to the presence of ``detectable but unobservable'' modes. In Proposition~\ref{prop:id}, we additionally impose the excitation condition \eqref{eq:W}. Under this condition, the parameter $\theta$ is locally identifiable over the observation interval $[t_1,T]$, and the unknown state can subsequently be reconstructed from its estimate. In other words, stronger identifiability conditions provide stronger guarantees on parameter and state reconstruction.

\noindent {\bf D6} We emphasize that the proposed framework does \emph{not} by itself provide an asymptotic convergence guarantee, as the state estimate is obtained by directly solving an optimization problem. On the other hand, in practical implementations, different constructions of the cost function $J$ are possible, such as formulations based on moving horizons or expanding horizons that progressively incorporate newly acquired information.

\noindent {\bf D7} For a linear time-varying system $(A(t), B(t), C(t))$, the associated identifiability is precisely the reconstructability of the initial condition $x(t_0)$. This property is equivalent to observability over some interval, without the requirement of uniformity with respect to time---as thoroughly discussed in \cite{WANetal24}.
\section{Example}
\label{sec:6}
The following autonomous example is used to illustrate the main ideas and theoretical results in a simple setting:
\begin{equation}
\begin{aligned}
    \dot x  = \begmat{-x_1 \\ -x_2 + x_1^2}
    , \quad 
    y  = x_2 + x_1^3.
\end{aligned}
\end{equation}
We have
$$
\begmat{\nabla_x h & \nabla_x (L_fh) }
=
\begmat{
3x_1^2 & -9x_1^2 + 2x_1\\ 1& -1
},
$$
which loses rank along the lines in the set 
$
\Omega_l:=\{x\in \rea^2: x_1=0 ~\mbox{or}~ x_1 ={1\over3}\},
$
independent of $x_2$. To further verify the injectivity of the differential observability map, we compute
$
L_f^2h=x_2-3x_1^2+9x_1^3.
$
Defining
$
\mathcal O_2(x):=\col\big(h(x),L_fh(x),L_f^2h(x)\big),
$
we have
$$
2h(x)+3L_fh(x)+L_f^2h(x)=2x_1^3.
$$
Hence, $x_1$ is uniquely determined by $\mathcal O_2(x)$, and subsequently
$x_2=h(x)-x_1^3$. Thus, $\mathcal O_2$ is injective on $\rea^2$.
Moreover, its Jacobian has full column rank for all $x_1\neq0$.
Therefore, the given system is strongly differentially observable almost globally in $\rea^2$, and satisfies Assumption~\ref{assm:2}.

The flow is given by
$$
X(x,t;s) 
=
\begmat{
e^{-(s-t)}x_1
\\
e^{-(s-t)}x_2 + x_1^2\left( e^{-(s-t)}  - e^{-2(s-t)}\right)
}.
$$
To illustrate the proposed construction, we simply select
$B= \mathbf{1}_3$, $H(h(x),\cdot) = h(x)$, and
$A = \diag\{-1,-2,-3\}$. Then, a feasible solution of the PDE
\eqref{pde} is 
$\phi(x,t) = P(t) \col(x_2, x_1^2, x_1^3)$
with
$$
P(t) =  
\begmat{
te^t & e^t+te^t - e^{2t} & {e^{3t}-e^t \over 2}
\\
e^{2t} - e^t & e^{2t} - e^t - te^{2t} & e^{3t} - e^{2t}
\\
{e^{3t}- e^t \over 2} &  {e^{3t}-e^t\over 2}-e^{3t}+e^{2t} & te^{3t}
}
$$
which is full rank for $t>0$. The mapping
$\phi(\cdot,t)$ is injective for $t>0$, with a left inverse
given by\footnote{Compared with the left inverse used in \cite{CDC}, the present construction avoids singularity at $x_1=0$. It also yields a regression model that is linear in $\theta$.}
\begin{equation}
x=\phil(z,t)
=
\begmat{
{\sqrt[3]{{\bf e}_3^\top P^{-1}(t)z}}
\\[.3cm]
{\bf e}_1^\top P^{-1}(t)z
}.
\end{equation}

The observer implementation starts at $t_0=0.1$ s until $t_f =1$ s. The output is uniformly sampled with a sampling period of $0.1$ ms. We can verify that
$z=\phi(x,t)$ admits the dynamics
$
\dot z=\beta(h(x),t),
$
with $\beta(h(x),t)=e^{-At}Bh(x)$. Accordingly, we introduce the
dynamic extension
$
\dot\zeta=\beta(y,t), \; \zeta(t_0)=\zeta_0,
$
where $\zeta_0\in\rea^{n_z}$ can be arbitrarily selected.

By substituting
$x=\phil(\zeta+\theta,t)$
into the output function $y=h(x)$, we obtain the parameterized
\emph{linear} regressor
\begin{equation}
y =
\left({\bf e}_1+{\bf e}_3\right)^\top
P^{-1}(t)\big(\zeta(t)+\theta\big).
\end{equation}
This relation holds for $t>t_0$, where
\(
\theta:=\phi(x(t),t)-\zeta(t)\in\rea^3
\)
is constant. Clearly, at a single time instant $t$, this scalar
regression is insufficient to uniquely determine $\theta\in\rea^3$.
This motivates using measurements collected over a time interval to
estimate $\theta$, as illustrated in the simulations below.

First, we illustrate the identifiability result established in Section~\ref{sec:4a}, showing that at a single time instant the parameter vector $\theta$ cannot be uniquely determined in $\mathbb{R}^{n_z}$. Specifically, we consider the time instant $t' = 0.2$~s and visualize the corresponding cost function $J(\theta) = |y(t') - h\circ \phil(\zeta(t') + \theta, t')|^2$ in Fig. \ref{fig:J}, where $\theta_3$ is fixed to its true value $\theta_{\star,3}$. As shown in the figure, there exist multiple parameter values distinct from $\theta_\star$ that achieve zero cost, i.e., $J = 0$. In the second simulation, the unknown parameter vector $\theta$ is estimated using a single batch optimization over the entire data record collected on the interval $[0.1,\,1]$~s. 
In view of {\bf R5}, we use the unconstrained solver MATLAB \texttt{fminsearch} to the optimization problem only once, at $t_f=1$~s. The results are shown in Fig.~\ref{fig:pebo_batch}. For completeness, we also display the reconstructed state
\(
\hat{x}=\phil(\zeta+\hat{\theta},t),
\)
which is not computed in real time. The reconstructed state closely matches the true state, while the final parameter estimation errors $\tilde{\theta}_i=\hat{\theta}_i-\theta_i$ are approximately of the order of $5\times10^{-4}$. In the next simulation, the parameter $\theta$ is estimated using an expanding-horizon strategy. At each update time $t_i>t_0$, spaced by $0.1$~s, we solve
\begin{equation*}
\hat\theta_i \in \underset{\vartheta}{\arg\min}\;
\int_{t_0}^{t_i}
\left|y(s)-h\circ\phil(\zeta(s)+\vartheta,s)\right|^2 ds.
\end{equation*}
All data collected on $[t_0,t_i]$, with $t_0=0.1$~s, are used with the previous estimate as a warm start. The results are shown in Fig.~\ref{fig:pebo_2}. The parameter estimation errors become negligible after the first update at $t=0.2$~s and remain so thereafter. Correspondingly, after a short initial transient, the reconstructed state closely matches the true state. For comparison, we also implement the online finite-convergence-time
least-squares (FCT-LS) parameter estimator proposed in
\cite{ORTetal26FCT}, with $\alpha=0.2$, $f_0=1$, $\beta_0=1$, and
$M=10$. The same simulation interval and sampling period are used.
As shown in Fig.~\ref{fig:pebo_fctls}, the parameter estimates
converge to their true values in finite time, yielding accurate
online state reconstruction after a short transient.

\begin{figure}[htbp]
    \centering
    \includegraphics[width=0.6\linewidth]{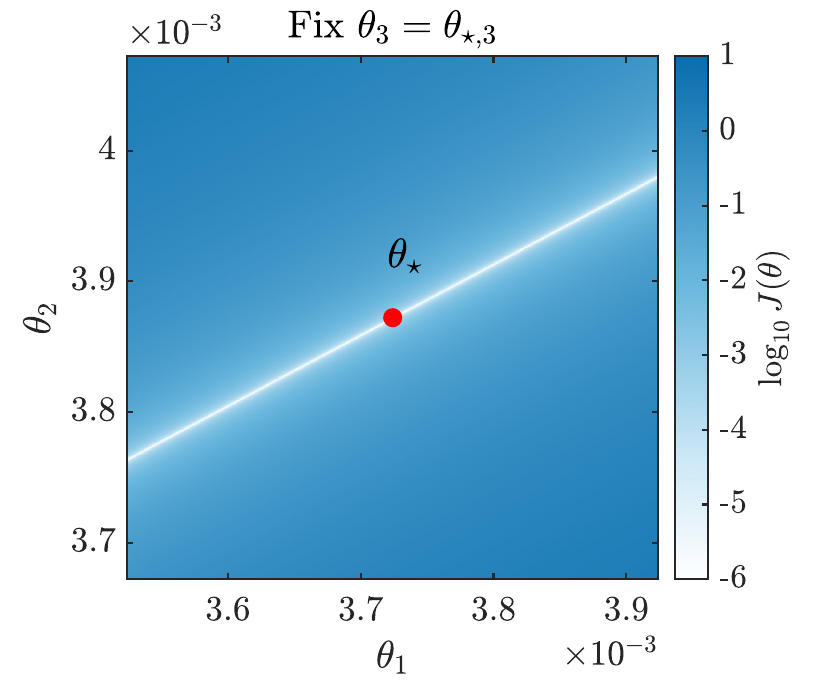}
    \caption{Single-time cost function illustrating lack of identifiability on $\rea^{n_z}$}
    \label{fig:J}
\end{figure}

\begin{figure*}[htbp]
\centering

\begin{minipage}[t]{0.49\textwidth}
\centering
\resizebox{0.98\linewidth}{!}{%
\begin{tikzpicture}
\begin{axis}[
every axis plot/.append style={line width=1.8pt},
  width=0.97\linewidth,
  height=4.7cm,
    xmin=0.1, xmax=1.0,     
  ymin=-1.1, ymax=1.1,     
  grid=both,
  xlabel={$t$ (s)},
  ylabel={states},
  legend style={font=\footnotesize, at={(0.35,0.98)}, legend columns=4, anchor=north west},
]
\addplot[
  color=blue!70!black!20, line width=3pt
] table [x=t, y=x1, col sep=comma] {figs/fig_state_recon.csv};
\addlegendentry{$x_1$}

\addplot[
  color=blue!70!black,
  dashed
] table [x=t, y=x1hat, col sep=comma] {figs/fig_state_recon.csv};
\addlegendentry{$\hat x_1$}

\addplot[
  color=red!80!black!20, line width=3pt
] table [x=t, y=x2, col sep=comma] {figs/fig_state_recon.csv};
\addlegendentry{$x_2$}

\addplot[
  color=orange!95!black,
  dashed
] table [x=t, y=x2hat, col sep=comma] {figs/fig_state_recon.csv};
\addlegendentry{$\hat x_2$}

\end{axis}
\end{tikzpicture}}

{\footnotesize (a) State reconstruction}
\end{minipage}
\hfill
\begin{minipage}[t]{0.49\textwidth}
\centering
\resizebox{0.98\linewidth}{!}{%
\begin{tikzpicture}
\begin{axis}[
  width=0.97\linewidth,
  height=4.7cm,
  grid=both,
  xlabel={$t$ (s)},
  ylabel={$\tilde\theta_i$},
  xmin=0.95, xmax=1.05,
  scaled y ticks=true,
  ytick scale label code/.code={$\times 10^{-4}$},
  legend style={font=\footnotesize, at={(0.02,0.98)}, anchor=north west},
]
\addplot[
  blue!70!black!20,
  only marks,
  mark=*,
  mark size=2.2pt,
  restrict x to domain=0.999:1.001
] table [x=t, y=theta1tilde, col sep=comma] {figs/fig_theta_tilde.csv};
\addlegendentry{$\tilde\theta_1$}

\addplot[
  red!80!black!30,
  only marks,
  mark=*,
  mark size=2.2pt,
  restrict x to domain=0.999:1.001
] table [x=t, y=theta2tilde, col sep=comma] {figs/fig_theta_tilde.csv};
\addlegendentry{$\tilde\theta_2$}

\addplot[
  green!80!black!30,
  only marks,
  mark=*,
  mark size=2.2pt,
  restrict x to domain=0.999:1.001
] table [x=t, y=theta3tilde, col sep=comma] {figs/fig_theta_tilde.csv};
\addlegendentry{$\tilde\theta_3$}
\end{axis}
\end{tikzpicture}}

{\footnotesize (b) Parameter estimation errors}
\end{minipage}

\caption{Simulation results: Using a single batch parameter estimate over the interval $[0.1,1.0]$\,s. The parameter vector $\theta$ is estimated only once at the final time $t_f = 1\,\mathrm{s}$ using all the collected data over the interval.} 
\label{fig:pebo_batch}
\end{figure*}
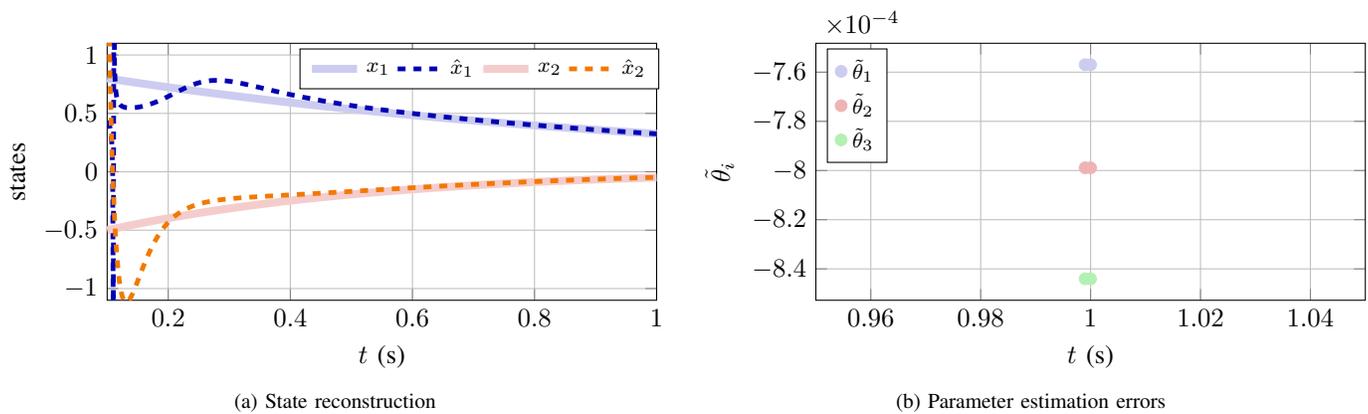

\begin{figure*}[htbp]
\centering
\begin{minipage}[t]{0.49\textwidth}
\centering
\resizebox{0.98\linewidth}{!}{%
\begin{tikzpicture}
\begin{axis}[
every axis plot/.append style={line width=1.8pt},
  width=0.97\linewidth,
  height=4.7cm,
    xmin=0.1, xmax=1.0,     
  ymin=-1.2, ymax=1.2,     
  grid=both,
  xlabel={$t$ (s)},
  ylabel={states},
  legend style={font=\footnotesize, at={(0.35,0.98)}, legend columns=4, anchor=north west},
]
\addplot[
  color=blue!70!black!20, line width=3pt
] table [x=t, y=x1, col sep=comma] {figs/gw_fig_state_recon.csv};
\addlegendentry{$x_1$}

\addplot[
  color=blue!70!black,
  dashed,
  line width=1.8pt,
  unbounded coords=discard
] table[
  x=t,
  col sep=comma,
  y expr={
    abs(\thisrow{x1hat})>2 ? nan : \thisrow{x1hat}
  }
] {figs/gw_fig_state_recon.csv};
\addlegendentry{$\hat x_1$}

\addplot[
  color=red!80!black!20, line width=3pt
] table [x=t, y=x2, col sep=comma] {figs/gw_fig_state_recon.csv};
\addlegendentry{$x_2$}

\addplot[
  color=orange!95!black,
  dashed
] table [x=t, y=x2hat, col sep=comma] {figs/gw_fig_state_recon.csv};
\addlegendentry{$\hat x_2$}

\end{axis}
\end{tikzpicture}}

{\footnotesize (a) State reconstruction}
\end{minipage}
\hfill
\begin{minipage}[t]{0.49\textwidth}
\centering
\resizebox{0.98\linewidth}{!}{%
\begin{tikzpicture}
\begin{axis}[every axis plot/.append style={line width=1.8pt},
  width=0.97\linewidth,
  height=4.7cm,
  grid=both,
  xlabel={$t$ (s)},
  ylabel={$\tilde\theta_i$},
  xmin=0.1, xmax=1,
  scaled y ticks=true,
  ytick scale label code/.code={$\times 10^{-4}$},
  legend style={font=\footnotesize, at={(0.72,0.58)}, anchor=north west},
]
\addplot[
  blue!90!black!50,
] table [x=t, y=theta1tilde, col sep=comma] {figs/gw_fig_theta_tilde.csv};
\addlegendentry{$\tilde\theta_1(t)$}

\addplot[
  red!90!black!50,
] table [x=t, y=theta2tilde, col sep=comma] {figs/gw_fig_theta_tilde.csv};
\addlegendentry{$\tilde\theta_2(t)$}

\addplot[
  green!90!black!50,
] table [x=t, y=theta3tilde, col sep=comma] {figs/gw_fig_theta_tilde.csv};
\addlegendentry{$\tilde\theta_3(t)$}

\end{axis}
\end{tikzpicture}}

{\footnotesize (b) Parameter estimation errors}
\end{minipage}

\caption{Simulation results: Using batch parameter estimates over expanding intervals $[0.1, t]$.}
\label{fig:pebo_2}
\end{figure*}
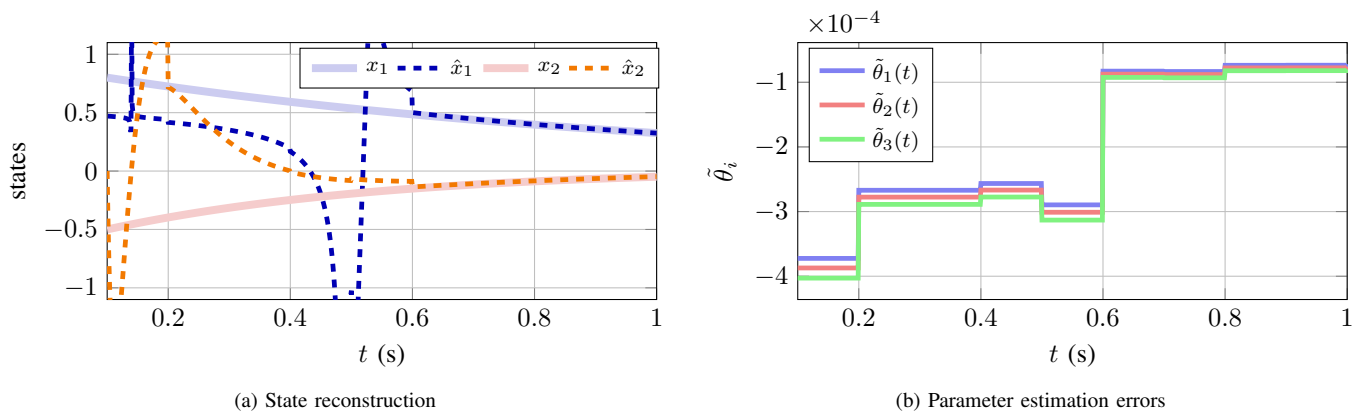

\begin{figure*}[!htp]
\centering

\begin{minipage}[t]{0.49\textwidth}
\centering
\resizebox{0.98\linewidth}{!}{%
\begin{tikzpicture}
\begin{axis}[
  every axis plot/.append style={line width=1.8pt},
  width=0.97\linewidth,
  height=4.7cm,
  xmin=0.1, xmax=1.0,
  ymin=-2, ymax=2,
  grid=both,
  xlabel={$t$ (s)},
  ylabel={states},
  legend style={
    font=\footnotesize,
    at={(0.35,0.98)},
    legend columns=4,
    anchor=north west
  },
]

\addplot[
  color=blue!70!black!20,
  line width=3pt
] table[
  x=t,
  y=x1,
  col sep=comma
] {figs/fig_lsd_fct_state.csv};
\addlegendentry{$x_1$}

\addplot[
  color=blue!70!black,
  dashed,
  line width=1.8pt,
  unbounded coords=discard
] table[
  x=t,
  y=x1hatFCT,
  col sep=comma
] {figs/fig_lsd_fct_state.csv};
\addlegendentry{$\hat x_1$}

\addplot[
  color=red!80!black!20,
  line width=3pt
] table[
  x=t,
  y=x2,
  col sep=comma
] {figs/fig_lsd_fct_state.csv};
\addlegendentry{$x_2$}

\addplot[
  color=orange!95!black,
  dashed,
  line width=1.8pt,
  unbounded coords=discard
] table[
  x=t,
  y=x2hatFCT,
  col sep=comma
] {figs/fig_lsd_fct_state.csv};
\addlegendentry{$\hat x_2$}

\end{axis}
\end{tikzpicture}}

{\footnotesize (a) Online state reconstruction}
\end{minipage}%
\hfill%
\begin{minipage}[t]{0.49\textwidth}
\centering
\resizebox{0.98\linewidth}{!}{%
\begin{tikzpicture}
\begin{axis}[
  every axis plot/.append style={line width=1.8pt},
  width=0.97\linewidth,
  height=4.7cm,
  xmin=0.1, xmax=1.0,
  grid=both,
  xlabel={$t$ (s)},
  ylabel={$\tilde\theta_i$},
  scaled y ticks=true,
  legend style={
    font=\footnotesize,
    at={(0.72,0.58)},
    anchor=north west
  },
]

\addplot[
  blue!90!black!50
] table[
  x=t,
  y=theta1tildeFCT,
  col sep=comma
] {figs/fig_lsd_fct_theta.csv};
\addlegendentry{$\tilde\theta_1(t)$}

\addplot[
  red!90!black!50
] table[
  x=t,
  y=theta2tildeFCT,
  col sep=comma
] {figs/fig_lsd_fct_theta.csv};
\addlegendentry{$\tilde\theta_2(t)$}

\addplot[
  green!90!black!50
] table[
  x=t,
  y=theta3tildeFCT,
  col sep=comma
] {figs/fig_lsd_fct_theta.csv};
\addlegendentry{$\tilde\theta_3(t)$}

\end{axis}
\end{tikzpicture}}

{\footnotesize (b) Online finite-time parameter estimation errors}
\end{minipage}

\caption{Online PEBO implementation using the finite-convergence-time
least-squares (FCT-LS) parameter estimator proposed in
\cite{ORTetal26FCT}. }
\label{fig:pebo_fctls}

\end{figure*}

\section{Conclusion}
\label{sec:7}
This paper studied the existence of PEBOs for general nonlinear systems of the form \eqref{NLsyst}. By separating the PEBO design problem into transformability and identifiability, we provided explicit sufficient conditions under which each property holds. In particular, we showed that the associated PDE admits a solution and that its injectivity can be generically guaranteed by the backward distinguishability of the plant, and that the identifiability of the parameter vector $\theta$ is closely related to suitable notions of observability. As a result, the paper provides a systematic framework for analyzing and designing PEBOs beyond case-by-case constructions. 

Nevertheless, the results are primarily of an existence nature, and numerical aspects and implementation issues remain important directions for future research. In particular,
\begin{itemize}
\item[\bf P1] The analytical construction of $\phi(\cdot,t)$ is not computationally efficient, as it requires explicit knowledge of the flow $X(x,t;s)$. Since the existence of such a mapping is guaranteed, it is natural to consider approximation-based approaches---such as deep learning models---to approximate the function $\phi(\cdot,t)$.

\vspace{.1cm}

\item[\bf P2] The optimization problem \eqref{eq:J} is highly nonconvex, and it would be interesting to study its properties ({\em e.g.}, the form of least squares) in order to develop more suitable optimization solvers for the proposed framework.

\end{itemize}

\section*{Acknowledgments}
This work was supported in part by the Natural Sciences and Engineering Research Council of Canada (NSERC) under Grant RGPIN-2024-0478, the Fonds de recherche du Québec (FRQ), and by the Programme PIED from Polytechnique Montr\'eal. B. Yi also acknowledges the support of Global Affairs Canada through the Faculty Mobility for Partnership Building Program (FMPBP).

\appendix

\section{Proof of Proposition \ref{prop:global_id}}
\label{sec:app_a}
\begin{proof}
(i) Under {\em Assumption}~\ref{assm:2}, Proposition~\ref{prop:2}
ensures that the mapping $\phi(\cdot,t)$ is injective on $\calx$
for all $t>t_\star$. Hence, the left inverse mapping
$\phil(\cdot,t)$ exists over this interval. Since
$t_c>\max\{t_0,t_\star\}$, the regressor
\eqref{eq:NLR} and the optimization problem \eqref{eq:opt} are
well defined over $[t_c,t]$ for any $t>t_c$.

Since $\calo_k(\cdot,t_c)$ is an injective immersion on
$\clo(\calx)$, by the stability of embeddings there exists a
sufficiently small $\Delta_o>0$ such that
$\calo_k(\cdot,\tau)$ remains an injective immersion on
$\clo(\calx)$ for all
\(
    \tau\in[t_c,t_c+\Delta_o];
\)
see Lemma~\ref{lem:embedding}.

We next show that $\phi(\cdot,t)$ is also an immersion in a
neighborhood of $t_c$. To this end, consider the variational
system of \eqref{NLsyst} along the trajectory $X(x,t_c;s)$,
\begin{equation}
\label{eq:var_global}
\begin{aligned}
    {d\over ds}\delta x
    =
    \underbrace{
    {\partial f\over\partial x}
    \big(X(x,t_c;s),s\big)
    }_{:=A_\delta(s)}
    \delta x,\;
    \delta y
    =
    \underbrace{
    {\partial h\over\partial x}
    \big(X(x,t_c;s)\big)
    }_{:=C_\delta(s)}
    \delta x ,
\end{aligned}
\end{equation}
with $\delta x\in\rea^n$ and $\delta y\in\rea^p$.
We denote by $\delta X(x,\delta x,t_c;s)$ the corresponding
state solution initialized from $\delta x$ at $t_c$.

At $s=t_c$, the successive output derivatives satisfy
\begin{equation}
\label{eq:var_O}
    \begmat{
        \delta y(t_c)\\
        \delta\dot y(t_c)\\
        \vdots\\
        \delta y^{(k)}(t_c)
    }
    =
    {\partial\calo_k\over\partial x}(x,t_c)\delta x .
\end{equation}
Since $\calo_k(\cdot,t_c)$ is an immersion,
${\partial\calo_k}/{\partial x}(x,t_c)$ has full column rank.
Hence, for every nonzero
$\delta x$,
\(
    {\partial\calo_k\over\partial x}(x,t_c)\delta x\neq0.
\)
Therefore, the variational system \eqref{eq:var_global} is
backward distinguishable in $\delta x$ around $t_c$.

This step is similar to the proof of \cite[Prop.~3.6]{AND}.
We now apply \cite[Theorem~3]{BERAND} to the variational
system \eqref{eq:var_global}. Since
$\tilde h(x,t):=H(h(x),t)$ and $H(\cdot,t)$ is a
diffeomorphism by Condition~\ref{cond:1}, the corresponding
output transformation preserves 
distinguishability for the variational system. Therefore, for the matrices $A$ and $B$
defined in \eqref{eq:AB}, there exists a set
$\cale_{\tt v}\subset\mathbb C^{n+1}$ of zero Lebesgue measure
such that, for any
$(\lambda_1,\ldots,\lambda_{n+1})
\in\mathbb C_{<0}^{n+1}\setminus\cale_{\tt v}$, the mapping
\begin{equation}
\label{eq:Tbar_tc}
\begin{aligned}
\bar T(\delta x,x,t_c)
:=
\int_0^{t_c}
&e^{A(t_c-s)}B
{\partial\tilde h\over\partial x}
\big(X(x,t_c;s),s\big)
\\[-1mm]
&\times
\delta X(x,\delta x,t_c;s)\,ds
\end{aligned}
\end{equation}
is injective in $\delta x$.

Applying \eqref{eq:delta_X} in Lemma~\ref{lem:3}, we obtain
$$
\fitmath{$\displaystyle
\begin{aligned}
&\bar T(\delta x,x,t_c)
      \\
      & = 
\int_0^{t_c} \exp(A(t_c-s)) B {\partial \tilde h \over \partial x}(X(x,t_c;s), s) {\partial X \over \partial x}(x,t_c;s)  ds
\cdot \delta x
    \\
      & 
= {\partial T \over \partial x}(x,t_c)\cdot  \delta x,
\end{aligned}
$}
$$
where the last equality follows by differentiating the
definition of $T(x,t_c)$ with respect to $x$ under the
integral sign. Since $\bar T(\cdot,x,t_c)$ is linear and
injective in $\delta x$, it follows that $\rank \{\partial T /\partial x(x,t_c)\}=n$.
Recalling that $T(x,t)=e^{At}\phi(x,t)$, we conclude that $\rank\{\partial \phi /\partial x(x,t_c)\} =n$.
By continuity of ${\partial\phi}/{\partial x}$ and compactness
of $\clo(\calx)$, there exists a sufficiently small
$\Delta_\phi>0$ such that $\phi(\cdot,\tau)$ remains an
immersion on $\clo(\calx)$ for all
\(
    \tau\in[t_c,t_c+\Delta_\phi].
\)

Let
\(
    \Delta t:=\min\{\Delta_o,\Delta_\phi\}.
\)
For any fixed $t>t_c$, define
\[
    \bar t:=\min\{t,t_c+\Delta t\}.
\]
Clearly, $\bar t>t_c$. Moreover, for all
$\tau\in[t_c,\bar t]$, the mapping
$\calo_k(\cdot,\tau)$ is injective and
$\phi(\cdot,\tau)$ is an immersion.

We now prove global identifiability on $\Theta(t)$.
Suppose, by contradiction, that there exists
$\theta'\in\Theta(t)$, with $\theta'\neq\theta$, such that
\begin{equation}
\label{eq:equal_regressor}
    y(\tau)=y_{\theta'}(\tau),
    \qquad
    \forall \tau\in[t_c,t],
\end{equation}
where, with a slight abuse of notation, we define
\[
    y_{\theta'}(\tau):=h(x_{\theta'}(\tau)),
    \qquad
    x_{\theta'}(\tau)
    :=
    \phil\big(\zeta(\tau)+\theta',\tau\big).
\]
Since $\theta'\in\Theta(t)$, we have
$\theta'\in\Omega_\tau$ for all $\tau\in[t_c,t]$, and hence
\[
    \zeta(\tau)+\theta'\in\phi(\calx,\tau).
\]
Consequently,
\begin{equation}
\label{eq:phi_candidate}
    \phi(x_{\theta'}(\tau),\tau)
    =
    \zeta(\tau)+\theta',
    \qquad
    \forall \tau\in[t_c,t].
\end{equation}

We next show that $x_{\theta'}(\cdot)$ is a trajectory of
the plant over $[t_c,\bar t]$. Since $\phi(\cdot,\tau)$ is an
injective immersion on this interval, its inverse on
$\phi(\calx,\tau)$ is locally $C^1$. Hence,
$x_{\theta'}(\tau)$ is differentiable on $[t_c,\bar t]$.
Differentiating \eqref{eq:phi_candidate} with respect to time
gives
\begin{equation}
\label{phi'1}
    \frac{\partial\phi}{\partial t}
    (x_{\theta'},\tau)
    +
    \frac{\partial\phi}{\partial x}
    (x_{\theta'},\tau)
    \dot x_{\theta'}
    =
    \dot\zeta(\tau)
    =
    \beta(y(\tau),\tau).
\end{equation}
From \eqref{eq:equal_regressor},
$y(\tau)=h(x_{\theta'}(\tau))$, and therefore
\begin{equation}
\label{phi'2}
    \beta(y(\tau),\tau)
    =
    \beta(h(x_{\theta'}(\tau)),\tau).
\end{equation}
On the other hand, since $\phi$ satisfies the PDE
\eqref{pde},
\begin{equation}
\label{phi'3}
    \frac{\partial\phi}{\partial t}
    (x_{\theta'},\tau)
    +
    \frac{\partial\phi}{\partial x}
    (x_{\theta'},\tau)
    f(x_{\theta'},\tau)
    =
    \beta(h(x_{\theta'}),\tau).
\end{equation}
Combining \eqref{phi'1}--\eqref{phi'3} yields
\begin{equation}
\label{eq:candidate_dyn_1}
    \frac{\partial\phi}{\partial x}
    (x_{\theta'},\tau)
    \big[
        \dot x_{\theta'}
        -
        f(x_{\theta'},\tau)
    \big]
    =0.
\end{equation}
Since $\phi(\cdot,\tau)$ is an immersion for all
$\tau\in[t_c,\bar t]$, the matrix
${\partial\phi}/{\partial x}(x_{\theta'},\tau)$ has full
column rank. Hence, \eqref{eq:candidate_dyn_1} implies
\begin{equation}
\label{eq:candidate_dyn}
    \dot x_{\theta'}
    =
    f(x_{\theta'},\tau),
    \qquad
    \forall \tau\in[t_c,\bar t].
\end{equation}

Therefore, both $x_{\theta'}(\cdot)$ and the true state
$x(\cdot)$ are trajectories of system~\eqref{NLsyst}.
Moreover, from \eqref{eq:equal_regressor}, their outputs
coincide over $[t_c,\bar t]$. Since the output is sufficiently
smooth, differentiating
\(
    h(x_{\theta'}(\tau))=h(x(\tau))
\)
along the two trajectories up to order $k$ gives
\begin{equation}
\label{eq:O_k_new}
    \calo_k(x_{\theta'}(\tau),\tau)
    =
    \calo_k(x(\tau),\tau),
    \qquad
    \forall\tau\in(t_c,\bar t).
\end{equation}
Since $\calo_k(\cdot,\tau)$ is injective on $\clo(\calx)$ for
all $\tau\in[t_c,\bar t]$, it follows that
\[
    x_{\theta'}(\tau)=x(\tau),
    \qquad
    \forall\tau\in(t_c,\bar t).
\]
Using \eqref{eq:phi_candidate} and
$\phi(x(\tau),\tau)=\zeta(\tau)+\theta$, we obtain
\[
\begin{aligned}
    \zeta(\tau)+\theta'
    &=
    \phi(x_{\theta'}(\tau),\tau)
    \\
    &=
    \phi(x(\tau),\tau)
    \\
    &=
    \zeta(\tau)+\theta,
\end{aligned}
\qquad
\forall\tau\in(t_c,\bar t).
\]
Hence,
\(
    \theta'=\theta,
\)
which contradicts $\theta'\neq\theta$. Since $[t_c,\bar t]\subseteq[t_c,t]$, Lemma~\ref{lem:id_interval}
implies that the regressor is also globally identifiable on
$\Theta(t)$ over $[t_c,t]$.

(ii) We now consider the optimization problem
\eqref{eq:opt}--\eqref{eq:J}. First, since
$\phi(x(\tau),\tau)=\zeta(\tau)+\theta$, the true parameter
satisfies
\(
    \theta\in\Omega_\tau,
    \;
    \forall\tau\in[t_c,t].
\)
Hence,
\(
    \theta\in\Theta(t).
\)
Moreover,
\[
\begin{aligned}
        J(\hat\theta) & \geq0,
    \qquad
    \forall\hat\theta\in\Theta(t)
    \\
    J(\theta) & =0.
\end{aligned}
\]
Thus, the minimum value of $J$ over $\Theta(t)$ is zero.

Suppose that another
$\hat\theta\in\Theta(t)$ also attains this minimum, i.e.,
$J(\hat\theta)=0$. Since the integrand in \eqref{eq:J} is
continuous and nonnegative, this implies
\[
    y(\tau)
    -
    h\circ\phil\big(
        \zeta(\tau)+\hat\theta,\tau
    \big)
    =0,
    \qquad
    \forall\tau\in[t_c,t].
\]
By the global identifiability established in part (i), we must
have
\(
    \hat\theta=\theta.
\)
Consequently, $\theta$ is the unique global minimizer of
\eqref{eq:opt}--\eqref{eq:J}.
\QED
\end{proof}

\section{Auxiliary lemmata}

In this appendix, we provide some lemmata that are used in the proof.

\begin{definition}\rm
We say that a given property of a map $\varphi(\cdot,0)$ is stable for any homotopy $\varphi(\cdot,t)$ of $\varphi(\cdot,0)$, there exists some $\epsilon>0$ such that for every $0<t<\epsilon$ also has that property. 
\end{definition}

\begin{lemma}
\textup{\cite[page 35]{GUIPOL}} 
\label{lem:embedding}\rm
Embeddings of a compact manifold $X$ into a manifold $Y$ are stable under homotopy.
\end{lemma}

\begin{lemma}
\label{lem:id_interval}\rm
Let $\cali_1\subseteq\cali_2$ be two time intervals. If a
regressor
\[
    \calh(t)=\Phi(\theta,t)
\]
is globally identifiable on $\Theta$ over $\cali_1$, then it
is also globally identifiable on $\Theta$ over $\cali_2$.
\end{lemma}

\begin{proof}
    Global identifiability over $\cali_1$ means that, for any
    $\theta_a,\theta_b\in\Theta$ with $\theta_a\neq\theta_b$,
    there exists an instant $t'\in\cali_1$ such that
    \[
        \Phi(\theta_a,t')\neq\Phi(\theta_b,t').
    \]
    Since $\cali_1\subseteq\cali_2$, we also have
    $t'\in\cali_2$. Hence, by definition, the regressor is
    globally identifiable on $\Theta$ over $\cali_2$.
\end{proof}

\begin{lemma}\rm
\label{lem:3}
The solution of the variational system \eqref{eq:ltv} along $X(x,t;s)$ generated from \eqref{NLsyst} is given by
\begin{equation}
\label{eq:delta_X}
    \delta X(x, \delta x, t;s) = {\partial X \over \partial x}(x,t;s)\delta x.
\end{equation}
\end{lemma}
\begin{proof}
 One has
$$
\fitmath{$\displaystyle
\begin{aligned}
{d\over ds} \left[{\partial X \over \partial x}(x,t;s) \delta x \right]
& = 
{\partial \over \partial x}\left( {\partial X \over \partial s} (x,t;s) \delta x \right)
\\
& = 
{\partial \over \partial x} \left(f(X(x,t;s),s) \delta x \right)
\\
& =
{\partial f \over \partial x}(X(x,t;s),s) 
\left[{\partial X \over \partial x}(x,t;s)\delta x\right].
\end{aligned}
$}
$$
In other words, we have
\begin{equation}
\label{eq:pX}
\begin{aligned}
  {d\over ds} \left[{\partial X \over \partial x}(x,t;s) \delta x \right] & = A_\delta(s) \left[{\partial X \over \partial x}(x,t;s) \delta x \right]  
  \\
  {\partial X \over \partial x}(x,t;t) \delta x
  & = \delta x.
\end{aligned}
\end{equation}
By comparing the ``LTV'' systems \eqref{eq:ltv} and \eqref{eq:pX}, we obtain that the solution $\delta X(x, \delta x, t;s)$ is \eqref{eq:delta_X}. \QED
\end{proof}

\end{document}